\documentclass[a4paper]{amsart}

\usepackage{amssymb}
\usepackage{amsthm}  
\usepackage{amsmath} 
\usepackage{amscd} 
\usepackage{float}
\usepackage[all]{xypic}
\usepackage{url}
\usepackage{hyperref}
\usepackage{color}
\hypersetup{linktocpage}

\theoremstyle{plain}

\newtheorem{Theorem}{Theorem}[section]
\newtheorem{Proposition}[Theorem]{Proposition}
\newtheorem{Lemma}[Theorem]{Lemma}
\newtheorem{Corollary}[Theorem]{Corollary}

\theoremstyle{definition}
\newtheorem{Definition}{Definition}
\newtheorem{Example}{Example}

\theoremstyle{remark}
\newtheorem{Remark}{Remark}

\newcommand{\om}{\omega}

\newcommand{\ba}{\mathcal{G}}

\newcommand{\fg}{\frak g}
\newcommand{\fp}{\frak p}
\newcommand{\fk}{\frak k}

\newcommand{\fh}{\frak h}
\newcommand{\fn}{\frak n}
\newcommand{\fa}{\frak a}
\newcommand{\fb}{\frak b}
\newcommand{\infi}{\frak{inf}}
\newcommand{\sinf}{\Sigma^{\frak{inf}(u_0)}}
\newcommand{\Ad}{{\rm Ad}}
\newcommand{\Fl}{{\rm Fl}}
\newcommand{\Aut}{{\rm Aut}}
\newcommand{\id}{{\rm id}}

\newcommand{\gr}{{\rm gr}}
\newcommand{\lz}{[\![}
\newcommand{\pz}{]\!]}

\newcommand{\Imm}{{\rm im}}

\begin{document}
\title{On automorphisms with natural tangent actions on homogeneous parabolic geometries}
\author{Jan Gregorovi\v c and Lenka Zalabov\' a}
\address{Jan Gregorovi\v c \\ Department of Mathematics\\ Aarhus University  \\ Ny Munkegade 118 \\  Aarhus C 8000 \\ Denmark \\and\\Department of Mathematics\\ and Statistic\\ Faculty of Sciences\\  Masaryk University\\ Kotl\' a\v rsk\' a 2\\ Brno 611 37\\ Czech Republic\\ jan.gregorovic@seznam.cz}
\address{Lenka Zalabov\' a \\ Institute of Mathematics and Biomathematics \\ Faculty of Science \\ University of South Bohemia \\ Brani\v sovsk\' a 31 \\ \v Cesk\' e Bud\v ejovice 370 05 \\  Czech Republic \\ lzalabova@gmail.com}
\thanks{ First author supported by the project CZ.1.07/2.3.00/20.0003 of the Operational Programme Education for
Competitiveness of the Ministry of Education, Youth and Sports of the Czech
Republic. Second author supported by the grant P201/11/P202 of the Czech Science Foundation (GA\v CR)
}
\maketitle
\begin{abstract} 
We consider automorphisms of homogeneous parabolic geometries with a fixed point. Parabolic geometries carry the distinguished distributions and we study those automorphisms which enjoy natural actions on the distributions at the fixed points. We describe the sets of such automorphisms on homogeneous parabolic geometries in detail and classify, whether there are non--flat homogeneous parabolic geometries carrying such automorphisms. We present two general constructions of such geometries and we provide complete classifications for the types $(G,P)$ of the parabolic geometries that have $G$ simple and the automorphisms are of order $2$.
\end{abstract}

\section*{Introduction}

In this article, we will investigate automorphisms of Cartan geometries $(\pi: \ba\to M,\om)$ of type $(G,P)$ over a connected smooth manifold that have fixed points and a natural tangent actions at that points. The crucial fact we will use is that the choice of a point $u\in \ba$ in the fiber over the fixed point allows us to identify such automorphisms with elements of $P$. In particular, we can view $u$ as a (higher order) frame of $T_{\pi(u)}M$, and the elements of $P$ then prescribe the action of the automorphisms in this frame. The reduction to the first order frame of $T_{\pi(u)}M$ is then represented by the Lie group homomorphism $P\to Gl(\fg/\fp)$ induced by the adjoint action. We would like to investigate automorphisms with natural tangent actions that do not depend on the choice of a preferred frame. In other words, we are interested in the elements that are mapped into the center of the image of the above Lie group homomorphism $P\to Gl(\fg/\fp)$.

The affine symmetric spaces are our main motivation for the investigation of such automorphisms. Indeed, the unique non--trivial choice of a natural tangent action for the affine Cartan geometries  is $-\id_{\fg/\fp}$, i.e., the action of the geodesic symmetry. Thus the investigation of such automorphisms for  Cartan geometries of an arbitrary type will give natural generalizations of symmetric spaces.

In this article, we will focus on regular parabolic geometries. These are Cartan geometries of type $(G,P)$, where $G$ is semisimple and $P$ is parabolic subgroup of $G$, 
and which are equivalent to a specific underlying geometric structures on $M$. The geometric data, which describe the underlying geometric structures depend on the type $(G,P)$. But in all cases, there is a distinguished maximally non--integrable distribution $T^{-1}M$ in $TM$, which is invariant under the action of all automorphisms. The remaining geometric data then restrict the possible frames of $T^{-1}M$ and admissible (partial) connections on $T^{-1}M$. Thus we investigate the automorphisms with a fixed point $x\in M$ and a natural action on $T^{-1}_xM$ instead of the whole $T_xM$.

In particular, we investigate, which types of parabolic geometries admit non--flat homogeneous examples carrying automorphisms acting by a natural action on $T^{-1}_xM$. We will see that there are more possible natural actions than in the case of affine geometries.

In the first section, we recall basic facts on Cartan and parabolic geometries and their automorphisms. We distinguish the automorphisms according to their type $\Sigma$ (see Definition \ref{symm-flat}). In particular, we consider the types $\Sigma_J$, which distinguish automorphisms with the same natural action $J$ on $T^{-1}_xM$  (see Definition \ref{def2}). Further, we distinguish the types of automorphisms, which generalize the geodesic symmetries of the affine symmetric spaces to parabolic geometries, and we call them generalized symmetries (see Definition \ref{def3}). We also introduce all necessary notations and assumptions we use in this article.

We recall the structure of homogeneous parabolic geometries and the construction of them in the second section. We also describe the basic properties of automorphisms of arbitrary type $\Sigma$ on homogeneous parabolic geometries and discuss how to compute all of them (see Theorem \ref{lemma3.10}).

In the third section, we focus on the automorphisms of type $\Sigma_J$. We describe them in detail and classify all possible natural actions $J$ on $T^{-1}_xM$ that can arise (see Proposition \ref{lambda-action}). The main result of this section is the Theorem \ref{2.5}, which describes the set of all automorphisms of type $\Sigma_J$.

The fourth section contains  the results, which relate the existence of automorphisms of type $\Sigma_J$ to the existence of generalized symmetries (see Theorem \ref{homsym-central}). This allows us to describe the curvature restrictions on these geometries in detail. Conversely, we give the restriction on the number of the generalized symmetries given by the existence of a non--trivial harmonic curvature (see Theorem \ref{curvrest}).

We describe important classes of generalized symmetries in the fifth section. We show that these classes are in many cases closely related to symmetric spaces, and that the generalized symmetries in these cases are often coverings of the geodesic symmetries.

We describe two general constructions of non--trivial examples of homogeneous parabolic geometries with generalized symmetries in the sixth section. Moreover, we use the constructions to give several remarkable examples explicitly.

In the last section, we provide tables with the complete classifications for the types $(G,P)$ of the parabolic geometries that have $G$ simple, the automorphisms are of order $2$ and that admit non--flat homogeneous parabolic geometries with these types of generalized symmetries.

\section{Structure of parabolic geometries and their automorphisms}

{\bf Basic facts on Cartan geometries.}
Before we start to study the parabolic geometries, let us recall several general facts on Cartan geometries of an arbitrary type from \cite{parabook,S} we will use later. In particular, we discuss the automorphisms with a natural action on $TM$.

In the article, we will always consider a Cartan geometry $(\pi: \ba\to M,\om)$ of type $(G,P)$ over a connected smooth manifold $M$ with the curvature (function) $\kappa:\ba\to \wedge^2 (\fg/\fp)^*\otimes \fg.$ We denote by $\Aut(\ba,\om)$ the automorphism group of the Cartan geometry $(\pi: \ba\to M,\om)$. Let us summarize here the crucial facts about $\Aut(\ba,\om)$ we mentioned in the Introduction.

\begin{enumerate} \item[{\bf Fact}]
The choice of $u\in \ba$ provides the inclusion $\Aut(\ba,\om)\to \ba$ given by the evaluation of each automorphisms at $u$.
\end{enumerate}

In particular, the transition $s\in P$ between the frames $u$ and $\phi(u)=u\cdot s$ completely describes the action of the automorphisms $\phi$ with the fixed point $\pi(u)$. Since the Cartan connection $\om$ identifies $TM$ with $\ba\times_{\underline{\Ad}(P)}\fg/\fp$, where $\underline{\Ad}$ is the truncated adjoint action on the quotient, the automorphism $\phi$ acts as $$T_{\pi(u)}\phi.\lz u,X+\fp\pz=\lz u,\underline{\Ad}_s(X+\fp)\pz$$ on $T_{\pi(u)}M$.

The consequence of the first Fact is that we can identify the Lie subgroup in $\Aut(\ba,\om)$ of all automorphisms $\phi$ with fixed point $\pi(u)$ with the Lie subgroup $A_u\subset P$ of all transitions between $u$ and $\phi(u)$. Then $\underline{\Ad}$ is a Lie group homomorphism of $A_u$ to $Gl(\fg/\fp)$, and the kernel of $\underline{\Ad}$ describes the automorphisms with the same tangent action in the frame $u$. If we change the frame $u$ to $u\cdot p$ for $p\in P$, then $p^{-1}sp$ is the element of $A_{up}$ representing the action of automorphism $\phi$. We will use the following notion for automorphisms with the distinguished action in some frame in the fiber over $\pi(u)$.

\begin{Definition} \label{symm-flat}
Let $(\pi: \ba\to M,\om)$ be a Cartan geometry of type $(G,P)$ and $\Sigma$ be a subset of $P$ stable under the conjugation by elements of $P$. We say that an automorphism represented by $s\in A_u$ as above is an \emph{automorphism of type $\Sigma$} at $\pi(u)$ if $s\in \Sigma$. For a fixed $u$, we call the elements of $A_u$ automorphisms, too.
\end{Definition}

Since the set $\Sigma$ is stable under the conjugation, the definition is independent of the actual choice of the frame $u$. On the other hand, the automorphisms of the same type $\Sigma$ at $\pi(u)$ can have different tangent actions in the given frame $u$, because a priory $\underline{\Ad}_s\neq \underline{\Ad}_{p^{-1}sp}$. However, there are automorphisms, which share the same tangent action independently of the choice of the frame $u$. These are given by the elements of the center $Z(\underline{\Ad}(P))$ of the image of $\underline{\Ad}$. Such tangent actions are natural, and we can consider the types of automorphisms given by the preimage of elements in $Z(\underline{\Ad}(P))$. Let us demonstrate this on affine Cartan geometries.

\begin{Example}
Let us consider affine geometries, i.e., Cartan geometries of type $(\mathbb{R}^n\rtimes Gl(n,\mathbb{R}),Gl(n,\mathbb{R}))$. Then $\mathbb{R}^n$ can be identified with $\fg/\fp$ in a canonical way, and $\underline{\Ad}$ is injective. In particular, the types of automorphisms with natural tangent actions correspond  to $\Sigma=\{a\cdot \id_{\mathbb{R}^n}\}\subset Z(Gl(n,\mathbb{R}))$ for $a\in \mathbb{R}$. Since the automorphisms of the Cartan geometry are exactly the affine transformations, they have to preserve the torsion $T$ and the curvature $R$ of the corresponding affine connection. It is clear that the automorphism of type $\{a\cdot \id_{\mathbb{R}^n}\}$ at $\pi(u)$ act as $\frac1a$ on $T(\pi(u))$ and as $\frac{1}{a^2}$ on $R(\pi(u))$. Thus if we want to consider non--flat affine geometries with $a\cdot \id_{\mathbb{R}^n}\in A_u$, then the only possible values of $a$ are $\pm 1$. So the only non--trivial affine transformations we can consider are the geodesic symmetries corresponding to $a=-1$. The well known fact is the following:

\begin{enumerate}\item[{\bf Fact}]
A homogeneous affine geometry such that the geodesic symmetry is at one (and thus at each) point of $M$ a (global) affine transformation is a symmetric space.
\end{enumerate}

Let us remark that for homogeneous affine geometries, the vanishing of the torsion is sufficient for $-\id_{\mathbb{R}^n}$ to be an affine transformation.
\end{Example}

\noindent
{\bf Structure of the type $(G,P)$ of parabolic geometries.}
Before we discuss the automorphisms of parabolic geometries with a natural action on the distinguished distribution $T^{-1}M$, let us make some assumption and recall the necessary technicalities on the structure of parabolic geometries. We follow here the book \cite{parabook}.

Let $\fg$ be the Lie algebra (or the underlying real Lie algebra) of the real (complex) semisimple Lie group $G$. Then we will denote by $\alpha_i$ the simple restricted roots of $\fg$, and we consider their ordering according to \cite[Appendix B]{parabook}. There is the following crucial fact about the structure of the type $(G,P)$ of parabolic geometries.

\begin{enumerate}\item[{\bf Fact}]
There is a bijection between standard parabolic subalgebras of $\fg$ and subsets $\Xi$ of simple restricted roots (which are stable under the complex conjugation $\sigma^*$).
Each parabolic subalgebra of $\fg$ is conjugated to a standard one.
\end{enumerate}

In this article, we will assume (without loss of generality) that the Lie algebra $\fp$ of $P$ is the standard parabolic subalgebra of $\fg$. Further, we will assume that the maximal normal subgroup of $G$ contained in $P$ is trivial, i.e., that the pair $(G,P)$ is effective. For such parabolic geometries, $\id_\ba$ is the unique automorphism acting as $\id$ on $M$.

We use the notation $P_{i_1,\dots,i_j}$ for the parabolic subgroup corresponding to the set $\Xi:=\{\alpha_{i_1},\dots,\alpha_{i_j}\}$. The set $\Xi$ defines a functional $ht_{\Xi}$ on the weight space of $\fg$ given on simple restricted roots by $ht_{\Xi}(\alpha_i)=1$ if $\alpha_i\in \Xi$ and $ht_{\Xi}(\alpha_i)=0$ otherwise.
We will denote by $\mu^\fg$ the highest root of $\fg$ and define $k:=ht_{\Xi}(\mu^\fg)$. Then the structure of $\fg$ is described as follows:
\begin{enumerate}\item[{\bf Fact}]
\begin{enumerate}
\item The functional $ht_{\Xi}$ defines a $|k|$--grading $\fg_{-k}\oplus \dots \oplus \fg_0 \oplus \dots \oplus \fg_{k}$ of $\fg$ by $\fg_{\alpha}\subset \fg_{ht_{\Xi}(\alpha)}$ for the restricted root space $\fg_{\alpha}$ of the root $\alpha$.
\item There is  an induced $P$--invariant filtration of $\fg$ of the form $\fg^{-k}\supset  \dots \supset  \fg^0=\fp \supset  \dots \supset  \fg^{k}$ given by $\fg^i:=\fg_{i}\oplus \dots \oplus \fg_{k}$.
\item The subgroup $P$ is contained in the subgroup of filtration preserving elements of $G$, and contains its component of identity.
\end{enumerate}
\end{enumerate}

There always are distinguished subalgebras $\fg_0$, $\fp_+:=\fg^1$ and $\fg_{-}=\fg_{-k} \oplus \dots \oplus \fg_{-1}$.  The parabolic group $P$ decomposes into a semidirect product of the Lie group $G_0$ with the Lie algebra $\fg_0$, which consists of elements of $P$ preserving the grading of $\fg$, and $\exp(\fp_+)$, where $G_0$ naturally acts by the adjoint action. Moreover, the grading induces a finer structure of $P$, and we can write each $p \in P$ uniquely as $p=g_0 \exp Z_1 \dots \exp Z_k$, where $g_0 \in G_0$ and $Z_i \in \fg_i$ for $i=1, \dots, k$. 

\noindent
{\bf Automorphisms with a natural action on $T^{-1}M$.}
The consequence of the structure of the pair $(G,P)$ is that all objects of the parabolic geometry $(\pi: \ba\to M,\om)$ of type $(G,P)$ are filtered. In particular, there is a filtration $T^{i}M$ of $TM$ which is induced by the filtration of $\fg/\fp$, and there always is a distinguished distribution $T^{-1}M$, the smallest piece of the filtration of $TM$, which is isomorphic to $\ba\times_{\underline{\Ad}(P)} \fg^{-1}/\fp$. As we mentioned before, the distribution $T^{-1}M$ is the basic part of the geometric data describing the underlying geometric structure, so it is natural to study automorphisms with a fixed point $\pi(u)$ and a natural action on $T^{-1}_{\pi(u)}M$. In particular, the preimage of $Z(\underline{\Ad}|_{\fg^{-1}/\fp}(P))$ in $P=G_0\rtimes \exp(\fp_+)$ is $Z(G_0)\rtimes \exp(\fp_+)$, because $\exp(\fp_+)$ acts trivially on $\fg^{-1}/\fp$.

\begin{Definition}\label{def2}
Let $J\in Z(\underline{\Ad}|_{\fg^{-1}/\fp}(P))$. We denote by $\Sigma_J$ the preimage of $J$ in $P$, i.e., $$\Sigma_J=\{s\exp(Z)\in Z(G_0)\rtimes \exp(\fp_+): \underline{\Ad}_{s}=J\}.$$
\end{Definition}

Thus we can fix an arbitrary frame $u$ and use the Lie group $A_u$ to study the Lie group of automorphisms fixing $\pi(u)$, because the results for the automorphisms of type $\Sigma_J$ does not depend on the actual choice of $u$. In particular, this choice of $u$ allows us to identify the Lie algebra of infinitesimal automorphisms $T_e\Aut(\ba,\om)$ with the linear subspace $\infi(u)\subset \fg$ given by the trivialization of $T_e\Aut(\ba,\om)\subset T_u\ba\cong \fg$. We note that the curvature $\kappa(u)$ measures the difference between the Lie algebra structure of $\fg$ and $T_e\Aut(\ba,\om)$. In particular, we will view $\infi(u)$ as the Lie algebra of $\Aut(\ba,\om)$, and denote by $\frak a_u\subset \infi(u)$ the Lie algebra of $A_u$.

Let us look on the action of automorphisms of type $\Sigma_J$ on the neighborhood of the fixed point $\pi(u)$. If we consider the normal coordinates $\pi\circ \Fl_1^{\om^{-1}(X)}(u)$ for $X$ in some neighborhood of $0$ in $\fg_-$, then we can easily write down the action of automorphisms $s\exp(Z)\in A_u\cap \Sigma_J$. Precisely, the automorphism $s\exp(Z)$ acts as 
$$\pi\circ \Fl_1^{\om^{-1}(X)}(u) \mapsto \pi\circ \Fl_1^{\om^{-1}(X)}(us\exp(Z)).$$
Then $s$ provides the linear change of coordinates of $X$ in $\fg_-$, but the $\exp(\fp_+)$--part changes the normal coordinate system itself. Thus apart the automorphisms of type $\Sigma_J$, there is another distinguished class of automorphisms with natural properties.

\begin{Definition} \label{def3}
Let $(\pi: \ba\to M,\om)$ be a parabolic geometry of type $(G,P)$ and $s\in Z(G_0)$. We say that automorphisms of type $$\Sigma(s):=\{p^{-1}sp: p\in P\}$$ are \emph{s--symmetries}. We say that automorphisms of type $\bigcup_{s\in Z(G_0)} \Sigma(s)$ are \emph{generalized symmetries}.
\end{Definition}

We see immediately that $(\underline{\Ad}|_{\fg^{-1}/\fp})(\Sigma(s))=\{\underline{\Ad}_s\}$ holds, i.e., $s$--symmetries have a natural action on $T^{-1}M$. Thus we will always write the elements of $Z(G_0)$ in the form $s_J$ for $J=\underline{\Ad}_{s_J}\in Z(\underline{\Ad}|_{\fg^{-1}/\fp}(P))$. We will show in Proposition \ref{2.1} that there is bijection between the actions $J$ and elements $s_J$. Our main result of the fourth section states that on homogeneous parabolic geometries, the existence of automorphisms of type $\Sigma_J$ induces the existence of the generalized symmetries with the same action on $T^{-1}M$.

\noindent
{\bf Regularity and normality of parabolic geometries.}
Finally, let us add some further natural assumptions on the parabolic geometries we will investigate.

We will always assume that the parabolic geometry $(\pi: \ba\to M,\om)$ of type $(G,P)$ is regular. In particular, the distribution $T^{-1}M$ is maximally non--integrable with the symbol $\fg_-$ at all points of $M$. This means that $T^{-1}M$ generates the whole filtration $T^iM$ of $TM$, and it is reasonable to consider automorphisms with natural action just on $T^{-1}M$ and not on the whole $TM$. Indeed the action on $T^{-1}M$ of regular parabolic geometry induce the action on associated graded of $TM$.

Further, we will for simplicity assume that the parabolic geometry $(\pi: \ba\to M,\om)$ of type $(G,P)$ is normal, i.e., that $\partial^*\kappa=0$ holds for the curvature, where $\partial^*$ is the Kostant codifferential. This allows to study the harmonic curvature (function) $\kappa_H$,  which has its values in $ \ker(\partial^*)/\Imm(\partial^*)\cong H^2(\fg_-,\fg)$, instead of the whole curvature $\kappa$. In fact, $\kappa_h$ is a basic invariant of regular, normal parabolic geometries, because there is a differential operator $L$ such that $\kappa=L(\kappa_H)$, see \cite{BGG}. 

The $\fg_0$--module $H^2(\fg_-,\fg)$ decomposes into a sum of isotypical components ($\fg_0$--submodules) of $\fg_0$--dominant weights, which have multiplicity $1$. 
There are several possible ways, how to represent these isotypical components. We will use in this article the following ways:
\begin{enumerate}
\item Each component in $H^2(\fg_-,\fg)$ is uniquely given by the ordered pair $(i,j)$ such that the $\fg_0$--dominant weight $\mu$ has the lowest weight vector $X^{-\mu}$ mapping $\fg_{-\alpha_i}\wedge \fg_{s_{\alpha_i}(-\alpha_j)}$ to $\fg_{s_{\alpha_i}s_{\alpha_j}(-\mu^\fg)}$, where we denote by $s_{\alpha_i}$ the reflexion over the simple root $\alpha_i$.
\item Each component in $H^2(\fg_-,\fg)$ is uniquely given by the $\fg_0$--dominant weight $\tilde \mu$ in $H^2(\fp_+,\fg)$ given by the affine action of distinguished elements of the Hasse graph of the parabolic geometry $(G,P)$ on the highest weight $\mu^\fg$.
\end{enumerate}
The relation between the weights $\tilde \mu$ and $\mu$ is induced by the non--trivial isomorphism $H^2(\fp_+,\fg)\cong H^2(\fg_-,\fg)$ and there are algorithms for the computation of both $\mu$ and $\tilde \mu$ for all components of the harmonic curvature, see \cite{Si}.

Since $(\fg_-)^*=\fp_+$, we can equivalently view $X^{-\mu}$ as an element of $\fp_+\wedge \fp_+\otimes \fg$. Precisely, $X^{-\mu}=X^{\alpha_i} \wedge X^{s_{\alpha_i}(\alpha_j)}\otimes X^{s_{\alpha_i}s_{\alpha_j}(-\mu^\fg)}$, where $X^{\alpha_l}$ is a root vector in the root space $\fg_{\alpha_l}$.

\section{Homogeneous parabolic geometries and automorphisms of type $\Sigma$}
\label{s1}

{\bf Construction of homogeneous parabolic geometries.}
There is a general way how to construct all homogeneous parabolic geometries, see \cite{parabook} or \cite{M}. We recall the following definition.

\begin{Definition}
Let $(K,H)$ and $(G,P)$ be types of Cartan geometries. Then the pair of maps $(i,\alpha)$ is called an \emph{extension} of $(K,H)$ to $(G,P)$ if it satisfies the following conditions:

\begin{enumerate}
\item $i:H\to P$ is a Lie group homomorphism,
\item $\alpha: \frak{k}\to \frak{g}$ is a linear map extending $T_ei:\frak{h}\to \frak{p}$,
\item $\alpha$ induces an isomorphism $\frak{k}/\frak{h} \rightarrow \frak{g}/\frak{p}$ of vector spaces, 
\item $\Ad_{i(h)}\circ \alpha=\alpha \circ \Ad_h$ holds for all $h\in H$.
\end{enumerate}
\end{Definition}
Each extension $(i,\alpha)$ determines the \emph{extension functor} $\mathcal{F}_\alpha$ from the category of Cartan geometries of type $(K,H)$ to the category of Cartan geometries of type $(G,P)$.  
We have $\mathcal{F}_\alpha(\mathcal{B})=\mathcal{B}\times_{i(H)} P$ for the Cartan geometry $(\mathcal{B}\to M,\om)$ of type $(K,H)$, $\mathcal{F}_\alpha(\om)$ is the unique Cartan connection, which pulls back to $\om$ via the natural inclusion $\mathcal{B} \to \mathcal{F}_\alpha(\mathcal{B})$ and $\mathcal{F}_\alpha(\phi)$ for an automorphisms $\phi\in \Aut(\mathcal{B},\om)$ is the induced automorphism of the associated bundle $\mathcal{B}\times_{i(H)} P$.

There can be many extensions giving the same parabolic geometry. If we consider the natural inclusions $\Aut(\ba,\om) \to \ba$, then we obtain the inclusions  $A_{u_0}\subset P$ and $\infi(u_0)\subset \fg$   (which depend on the choice of $u_0$) which naturally satisfy all conditions of the extension for each choice of $u_0\in \ba$. We adopt the following terminology:

\begin{Definition}
We say that the homogeneous parabolic geometry $(\ba \rightarrow M,\om)$ of type $(G,P)$ \emph{is given by the extension $(i,\alpha)$ of $(K,H)$ to $(G,P)$ at $u_0\in \ba$}, if there is an isomorphism between parabolic geometries $\mathcal{F}_\alpha(K\to K/H,\om_K)$ and $(\ba \rightarrow M,\om)$ such that the class $\lz e,e\pz\in K\times_{i(H)}P$ corresponds to $u_0$.
\end{Definition}

Let us point out that the kernel of the evaluation map at $u_0$ coincides with the kernel of $i$ and consists of elements, which act as $\id$ on $K/H$. In particular, we will for simplicity assume that the pair $(K,H)$ is effective, $K\subset \Aut(\ba,\om)$ and $i$ injective and we will omit writing it.

\noindent
{\bf Automorphisms of type $\Sigma$ on homogeneous parabolic geometries.}
Let $(\pi:\ba\to M,\om)$ be a homogeneous parabolic geometry of type $(G,P)$. Then we can consider the $P$--invariant subset $\bigcup_{u\in \ba}(u,A_u)$ of the bundle $\ba\times P$ with the right action of $P$ given by $(u,q)p:=(up,p^{-1}qp)$. The homogeneity allows to describe all automorphisms of the given type $\Sigma$ at all points of $M$.

\begin{Proposition} \label{canonical-iso-homog}
Let $(\pi:\ba\to M,\om)$ be a homogeneous parabolic geometry of type $(G,P)$ and let $u_0\in \ba$. 
Then there is one--to--one correspondence between automorphisms of type $\Sigma$ at any fixed point of $M$ of the parabolic geometry $(\pi:\ba\to M,\om)$ and elements of $A_{u_0}\cap \Sigma$. In particular, $$\bigcup_{u\in \Aut(\ba,\om)(u_0)}(u,A_u\cap \Sigma)\cong  \Aut(\ba,\om)(u_0)\times (A_{u_0}\cap \Sigma).$$
\end{Proposition}
\begin{proof}
Clearly, $\phi^{-1}p \phi\in A_{u_0}$ holds for each $(\phi(u_0),p)\in \bigcup_{u\in \Aut(\ba,\om)(u_0)}(u,A_u)$. Thus, since $\phi$ is invertible and unique for each point of $\Aut(\ba,\om)(u_0)$, the second claim follows.

Since there always is $\phi\in \Aut(\ba,\om)$ such that $\pi(\phi(u_0))=x$ for any point $x\in M$, it follows that $\phi(A_{u_0}\cap \Sigma)\phi^{-1}$ is the set of automorphisms of type $\Sigma$ at $x$, because the type of automorphism $\Sigma$ is stable under conjugation.
\end{proof}

Let $(\pi:\ba\to M,\om)$ be a homogeneous parabolic geometry of type $(G,P)$ given by the extension $(i,\alpha)$ of effective $(K,H)$ to $(G,P)$ at $u_0\in \ba$. There is a natural question, how to find all automorphisms of a chosen type $\Sigma$ of this geometry, i.e., how to describe explicitly the set $A_{u_0}\cap \Sigma$? 

The trouble is that $K$ can be much smaller than the whole $\Aut(\ba,\om)$. For example, we can view the flat model of $(G,P)$ as an extension given by an inclusion of the maximal compact subgroup $K$ into $G$. In particular, $H\cap \Sigma$ can be empty, even if $A_{u_0}\cap \Sigma$ is non--empty.

\begin{Lemma} \label{lemma3.2}
Let $(\pi:\ba\to M,\om)$ be a homogeneous parabolic geometry of type $(G,P)$ given by the extension $(i,\alpha)$ of effective $(K,H)$ to $(G,P)$ at $u_0\in \ba$.
\begin{enumerate}
\item Let $h\in P$ be such that $\Ad_h(\alpha(\fk))\subset \alpha(\fk)$ and $h.\kappa(u_0)=\kappa(u_0)$ (for the action induced by $\Ad$--action). Then there is local automorphism $\phi$ of the parabolic geometry such that $\phi(u_0)=u_0h$. In particular, such $h$ induces an automorphism of $\fk$ and if $M$ is simply connected, then $h\in A_{u_0}.$
\item Denote $$\Sigma^\fk:=\{h\in \Sigma: h.\kappa(u_0)=\kappa(u_0),\Ad_h(\alpha(\fk))\subset \alpha(\fk)\}.$$
Then there exists an effective pair $(K',H')$ and an extension $(\alpha',i')$ of $(K',H')$ to $(G,P)$ such that $i'(H')\cap \Sigma=\Sigma^\fk$ and $\mathcal{F}_{\alpha'}(K'\to K'/H',\om_{K'})$ is simply connected covering of $(\pi:\ba\to M,\om)$. In particular, the parabolic geometries $\mathcal{F}_{\alpha'}(K'\to K'/H',\om_{K'})$ and $(\pi:\ba\to M,\om)$ are locally isomorphic and each automorphism of $\mathcal{F}_{\alpha'}(K'\to K'/H',\om_{K'})$ is a local automorphism of $(\pi:\ba\to M,\om)$.
\end{enumerate}
\end{Lemma}
\begin{proof}
If $h\in A_{u_0}$, then $\kappa(u_0)=\kappa(h(u_0))=\kappa(u_0h)=h.\kappa(u_0)$. Conversely, the following equality holds for the curvature: $$\kappa(u_0)(X,Y)=h.\kappa(u_0)(X,Y)=\Ad_h^{-1}([\alpha(X_h),\alpha(Y_h)]-\alpha([X_h,Y_h])),$$ where $X_h, Y_h$ are arbitrary elements of $\frak{inf}(u_0)$ congruent with $\Ad_h(X), \Ad_h(Y)$ modulo $\fp$, see \cite[Section 1.5.16]{parabook} for description of the curvature of homogeneous geometries. Since $\Ad_h(\alpha(\fk))\subset \alpha(\fk)$ and $\alpha$ is injective due to effectivity, we can take $X_h$ to be the preimage of $\Ad_h(\alpha(X))$ in $\fk$ with respect to $\alpha$ for each $X \in \fk$. Thus $\Ad_h([X,Y]-\alpha([X,Y]))=[\alpha(X_h),\alpha(Y_h)]-\alpha([X_h,Y_h])$. Comparing of the terms in the expression gives $[X,Y]_h=[X_h,Y_h]$ due to injectivity of $\alpha$, i.e. $h$ induces an automorphism of $\fk$. Let $K^c$ be connected, simply connected Lie group with the Lie algebra $\fk$. Then we can form the semidirect product $\bar K=K^c\rtimes \{h^k : k\in \mathbb{Z}\}$.

If $M$ is simply connected and $M=K^c/H^c$, then $H^c$ is connected and we can form the semidirect product $\bar H=H^c\rtimes \{h^k : k\in \mathbb{Z}\}$. Then if we define $\bar i(h',h^k):=i(\pi_H(h'))h^k$ for $\pi_H$ the covering homomorphism $H^c\to H$, then  $(\alpha,\bar i)$ is the extension of $(\bar K, \bar H)$ to $(G,P)$ giving the geometry $(\pi:\ba\to M,\om)$ at $u_0$. Thus if $M$ is simply connected, then $h\in A_{u_0}$.

If $M$ is not simply connected, then we consider the simply connected covering $K^c/H^c$ of $K/H$. Then $(\alpha,\bar i)$ is the extension of $(\bar K, \bar H)$ to $(G,P)$ giving the parabolic geometry $\mathcal{F}_{\alpha}(\bar K\to K^c/H^c,\om_{\bar K})$ on the simply connected covering. Clearly, $\mathcal{F}_{\alpha}(\bar K\to K^c/H^c,\om_{\bar K})$ is locally isomorphic to $(\pi:\ba\to M,\om)$ and thus the claim (1) follows.

The same construction can be done for the whole $\Sigma^\fk$ instead of $\{h^k : k\in \mathbb{Z}\}$, which implies the claim (2) directly. Precisely, we define $(K',H')$ to be the effective quotient of $(K^c\rtimes H',H^c\rtimes H')$, where $H'$ is the subgroup of $P$ generated by $i(H)$ and $\Sigma^\fk$, and we define $\alpha'=\alpha+\id_{\fh'}$ and $i'((h,h'))=i(\pi_H(h))h'$, which define the extension $(\alpha',i')$ of $(K',H')$ to $(G,P)$.
\end{proof}

In particular, we get the following consequence of the previous Lemma.

\begin{Corollary} \label{cor1.4}
For $\theta\in \Aut(\fg)$ such that $\theta(\fp)\subset \fp$, there is a pair $(G,P)$ satisfying all our assumptions such that $\theta=\Ad_{h}$ for $h\in P$.
\end{Corollary}

The distinguished case $\fk=\frak{inf}(u_0)$ allows to describe the set $A_{u_0}\cap \Sigma$ explicitly.

\begin{Theorem}\label{lemma3.10}
Let $(\pi:\ba\to M,\om)$ be a homogeneous parabolic geometry of type $(G,P)$. Then $$\sinf=\{h\in \Sigma: h.\kappa(u_0)=\kappa(u_0), \Ad_h(\frak{inf}(u_0))\subset \frak{inf}(u_0)\}$$ consists of local automorphisms of type $\Sigma$ and $A_{u_0}\cap \Sigma\subset \sinf$. In particular, if $M$ is simply connected, then $\sinf=A_{u_0}\cap \Sigma.$
\end{Theorem}

There is an algorithm for computation of the local infinitesimal automorphisms of homogeneous parabolic geometries, see \cite{M}. Then $\frak{inf}(u_0)$ consists of complete local infinitesimal automorphisms, and the above Theorem shows how to compute automorphisms of type $\Sigma$, which are contained in $\sinf$.

We will see in the next chapters, that the set $\sinf$ can be computed more easily in the case of $\Sigma_J$.

\section{Homogeneous parabolic geometries and automorphisms of type $\Sigma_J$}

{\bf Description of $\Sigma_J$ and structure of $Z(G_0)$.}
We focus here on types $\Sigma_J$ of automorphisms with a natural action on $T^{-1}M$. In particular, we describe the set $\Sigma_J$ for $J\in Z(\underline{\Ad}|_{\fg^{-1}/\fp}(P))$ in detail.

\begin{Proposition} \label{2.1}
For the pair $(G,P)$, there is a bijection between actions $J\in Z(\underline{\Ad}|_{\fg^{-1}/\fp}(P))$ and  elements $s_J\in Z(G_0)$. In particular, for each $J\in Z(\underline{\Ad}|_{\fg^{-1}/\fp}(P))$, there is a unique $s_J\in Z(G_0)$ such that $$\Sigma_J=\{s_J\exp(Z): Z\in \fp_+\}$$
and $J=\Ad_{s_J}|_{\fg_{-1}}$.
\end{Proposition}
\begin{proof}
The action of $\underline{\Ad}_s\in \underline{\Ad}|_{\fg^{-1}/\fp}(\Sigma_J)$ on $\fg^{-1}/\fp$ coincides with the action of $\Ad_s$ on $\fg_{-1}$. Let $s'_J \in Z(G_0)$ be another element such that $\Ad_{s'_J}=\Ad_{s_J}$ on $\fg_{-1}$. Then $\Ad_{s_J(s'_J)^{-1}}= \id$ holds on $\fg_-$, because $\fg_{-}$ is generated by $\fg_{-1}$. Since the geometry is effective, there is no ideal of $\fg$ contained in $\fg_0$. This means that $\Ad_{s_J(s'_J)^{-1}}= \id$ holds on the whole $\fg$ thanks to the duality $\fg_{-}^* = \fp_+$ and the fact $\fg_0 \subseteq  \fg_{-}^* \otimes \fg_{-}$. Thus the effectivity of $(G,P)$ implies $s_J=s'_J$ and the claim follows.
\end{proof}

Let us describe the elements of $Z(G_0)$ in the form of a functional on the space of weights.

\begin{Proposition}\label{lambda-action}
Let $\fg$ be a semisimple Lie algebra and $\fp$ its parabolic subalgebra, which does not contain any simple ideal of $\fg$. Then the following facts hold:

\begin{enumerate}
\item For each pair $(G,P)$ and each $s_J\in Z(G_0)$, 
there is a complex valued linear functional $\lambda_J$ given on the space of weights of $\fg$ given by 
$$\Ad_{s_J}(X^\alpha)=e^{\lambda_J(\alpha)}X^\alpha$$ for the restricted root $\alpha$.
In particular, the following conditions are satisfied:
\begin{enumerate}
\item If $X^\alpha\in \fg_0$, then $\lambda_J(\alpha)=0$.
\item If $\fg_\alpha$ is a real root space, then $\lambda_J(\alpha)  \in \mathbb{R}\mod \pi i$.
\item If $\fg_\alpha$ is a complex root space, then $\lambda_J(\alpha)$ and $\lambda_J(\sigma^* \alpha)$ are complex conjugate.
\end{enumerate}

\item Let $\lambda_J$ be a complex valued linear functional on the space of weights of $\fg$ satisfying (a)-(c). Then there is a pair $(G,P)$ satisfying all our assumptions with $s_J\in Z(G_0)$ satisfying 
$$\Ad_{s_J}(X^\alpha)=e^{\lambda_J(\alpha)}X^\alpha$$ 
for all restricted roots $\alpha$.

\item In particular, the unique eigenvalue of $s_J\in Z(G_0)$ on the component of the harmonic curvature $(i,j)$ is $e^{\lambda_J(\mu)}$ for the corresponding weight $\mu$, and if $X^{-\mu}=X^{\alpha_i} \wedge X^{s_{\alpha_i}(\alpha_j)}\otimes X^{s_{\alpha_i}s_{\alpha_j}(-\mu^\fg)}$, then $$\Ad_{s_J}(X^{-\mu})=e^{\lambda_J(\alpha_i)+\lambda_J(s_{\alpha_i}(\alpha_j))-\lambda_J(s_{\alpha_i}s_{\alpha_j}(\mu^\fg))}X^{-\mu}.$$
\end{enumerate}
\end{Proposition} 
\begin{proof}
We get the claim (1) just by the restriction of the adjoint representation of $G$ to $G_0$ and by the complete reducibility. Then the properties (a)--(c) trivially follows.
Conversely, if the conditions (a)--(c) are satisfied, then the formula in the claim (2) induces an automorphism of $\fg$, and the claim (2) follows from the Corollary \ref{cor1.4}, because $\Ad_{s_J}|_{\fg_0}=\id$ implies $s_J\in Z(G_0)$. The claim (3) is then a direct consequence of the description of components of the harmonic curvature.
\end{proof}

Clearly, the set $\Sigma_J$ is a union of $P$--conjugacy classes of suitable elements of $P$. We show that there are distinguished representatives of the $P$--conjugacy classes.

\begin{Lemma}\label{uscor}
There is a finite amount of $Z_i\in \fp_{fix}^{s_J}$ such that $$\Sigma_J=\Sigma(s_J)\cup \bigcup_i \Sigma(s_J\exp(Z_i)),$$ where $$\fp_{fix}^{s_J}:=\{Z\in \fp_+: \Ad_{s_J}(Z)=Z\}.$$  In particular, for $s_J\exp(Z)\in \Sigma_J$ there is $Y\in \fp_+$ such that $\exp(-Y)s_J\exp(Z)\exp(Y)\in s_J\exp(\fp_{fix}^{s_J})$ and $$\Sigma(s_J)=\{s_J\exp(C(-\Ad_{s_J}^{-1}(Y),Y)) : Y\in \fp_+\},$$
where $C$ denotes the application of the Baker--Campbell--Hausdorff (BCH) formula on the elements, see \cite[Theorem 4.29.]{KMS}.
\end{Lemma} 
\begin{proof}
For $s_J\in Z(G_0)$, there always is only a finite amount of $P$--conjugacy classes  $$\Sigma(s_J\exp(Z))=\{s_J\exp(-\Ad_{s_J}^{-1}Y)\exp(\Ad_{g_0}^{-1}Z)\exp(Y): g_0\in G_0, Y\in \fp_+\}$$ for $Z \in \fp_+$,
because there always is a finite amount of $\Ad(G_0)$--orbits in $\fp_+$.

Let us consider an element $\exp(Y) \in \exp(\fp_+)$. The conjugation of $s_J\exp(Z)$ by this element is of the form
\begin{align*}
&\exp(-Y)s_J\exp(Z)\exp(Y)=s_J\exp(-\Ad_{s_J}^{-1}Y) \exp(Z)\exp(Y).
\end{align*} 
Thus if $Z_f+Z'$ is the decomposition of $Z$ into $\fp_{fix}^{s_J}$ and the remaining eigenspaces of $\Ad_{s_J}$, then we can choose $Y:=(\id-\Ad_{s_J}^{-1})^{-1}Z'$ to change the representative, because $(\id-\Ad_{s_J}^{-1})$ is invertible on eigenspaces of $\Ad_{s_J}$ for eigenvalues different from $1$. Then the right hand side is of the form $$s_J\exp(\Ad_{s_J}^{-1}(\id-\Ad_{s_J}^{-1})^{-1}Z'+Z_f+Z'-(\id-\Ad_{s_J}^{-1})^{-1}Z'+ \cdots)=s_J\exp(Z_f+\cdots),$$
where $\cdots$ are given by the BCH--formula. Since $\cdots$ are in the higher order part of the filtration than $Z$, we can repeat the computation for $\cdots$. Then the result follows by induction, because the new terms appear in the part of the filtration containing $[Z_f,\cdots]$ due to BCH--formula, and thus vanish after $k$ steps of induction for $|k|$--graded geometry due to nilpotency.

The last formula then corresponds to the case $Z=0$.
\end{proof}

\noindent
{\bf Jordan decomposition of automorphisms of type $\Sigma_J$.}
It is clear from the Proposition \ref{lambda-action} that $s_J\in Z(G_0)$ are semisimple elements. This gives the following alternative conditions on $s_J$ being an element of $\Sigma_J^\fk:=(\Sigma_J)^\fk$.

\begin{Lemma}\label{altcon}
Let $(\pi: \ba\to M,\om)$ be a homogeneous parabolic geometry given by the extension $(i,\alpha)$ of effective $(K,H)$ to $(G,P)$ at $u_0 \in \ba$. Then the following claims hold:
\begin{enumerate}
\item The following facts are equivalent:
\begin{enumerate}
\item $\Ad_{s_J}.\kappa(u_0)=\kappa(u_0)$,
\item the components of the harmonic curvature such that $\lambda_J(\mu)\neq 0 \mod 2\pi i$ for the corresponding weight $\mu$ vanish at $u_0$.
\end{enumerate}
\item The following facts are equivalent:
\begin{enumerate}
\item $\Ad_{s_J}(\alpha(\fk))\subset \alpha(\fk)$,
\item $\alpha(\fk)$ splits into eigenspaces of $\Ad_{s_J}$ in $\fg$.
\end{enumerate}
\end{enumerate}
In particular, if one of the points (a) or (b) is satisfied for both claims (1) and (2), then $s_J\in \Sigma_J^\fk$.
\end{Lemma}
\begin{proof}
 (1) If $\Ad_{s_J}.\kappa(u_0)=\kappa(u_0)$, then the same holds for $\kappa_H$ and we know that $\lambda_J(\mu)= 0 \mod 2\pi i$ for the corresponding weight $\mu$ from Proposition \ref{lambda-action}. Conversely, the condition (b) implies $\Ad_{s_J}.\kappa_H(u_0)=\kappa_H(u_0s_J)=\kappa_H(u_0)$ and the claim follows from the general theory, because the splitting operator giving the whole curvature is $G_0$--equivariant, see \cite{BGG,parabook}. 

(2) The condition (b) clearly implies (a). Conversely, since $s_J$ is a semisimple element, there is a basis of $\alpha(\fk)$ consisting of eigenvectors of $\Ad_{s_J}$ and the claim follows.

In particular, $s_J\in \Sigma_J^\fk$ holds if the points (a) are satisfied for both claims (1) and (2) by the definition of $\Sigma_J^\fk$.
\end{proof}

We know from the Lemma \ref{uscor} that there is $u_0$ such that $s_J\exp(Z)$ is the Jordan decomposition of an element of $\Sigma_J$, i.e., $s_J$ is semisimple element, $\exp(Z)$ unipotent element and $\Ad_{s_J}(Z)=Z$. This has the following consequence.

\begin{Lemma} \label{h-minus}
Let $(\pi: \ba\to M,\om)$ be a homogeneous parabolic geometry of type $(G,P)$. Then there is $u_0\in \ba$ such that $s_J\in \Sigma_J^\fk$ for the extension $(i,\alpha)$ of effective $(K,H)$ to $(G,P)$ giving the geometry at $u_0$ with $\Sigma_J^\fk\neq \emptyset$.
\end{Lemma}
\begin{proof}
Let us choose the point $u_0$ such that $\Ad_{s_J}(Z)=Z$ holds for the element $s_J\exp(Z)\in \Sigma_J^\fk$. Then since $s_J\exp(Z)$ is the Jordan decomposition, $\Ad_{s_J}(\alpha(\fk))\subset \alpha(\fk)$ holds, i.e., the condition (a) of the claim (2) of the Lemma \ref{altcon} is satisfied. Moreover, since $s_J\exp(Z)\in \Sigma_J^\fk$, the condition (b) of the claim (1) of the Lemma \ref{altcon} is satisfied,  and the claim follows.
\end{proof}

 Let us adopt the following notation $$\fh_+:=\fh\cap \fp_+.$$

\begin{Lemma} \label{h-plus}
Let $(\pi: \ba\to M,\om)$ be a homogeneous parabolic geometry given by the extension $(i,\alpha)$ of effective $(K,H)$ to $(G,P)$ at $u_0 \in \ba$. Then $$H\cap \exp(\fp_+)=\exp(\fh_+)=\Sigma_{\id_{\fg^{-1}/\fp}}^\fk.$$ 
\end{Lemma}
\begin{proof}
Let us write $X_i+\dots$ for the element $X \in \fg^i$ such that $0\neq X_i \in \fg^i /\fg^{i+1} \simeq \fg_i$, and suppose $Z=Z_j+\dots$ for $\exp(Z)\in H\cap \exp(\fp_+)$. Then it follows from the general theory that there is an $\frak{sl}(2)$--triple $(X_{-j},A_{0},Z_{j})$ for $X_{-j}\in \fg_{-j}$ and $A_{0}\in \fg_0$, see \cite[Section X.3]{Kn96}. Since $j>0$, there is $X\in \alpha(\fk)$ such that $X=X_{-j}+\dots$ and we define $A:=(\id_{\fk}-\Ad_{\exp(Z)})(X)$, i.e., $\alpha(A)=-[Z_j,X_{-j}]+\cdots$ holds. 
Thus
$$(\id_{\fk}-\Ad_{\exp (Z)})^2(X)=[A_0,Z_j]+\dots=2Z+Y \in \fh_+,$$ 
where $Y\in \fp_+$ is in the higher order part of the filtration than $Z$. In particular, if $Y=0$, then $Z\in \fh$.
Otherwise, $(\exp(Z))^2\exp(-2Z-Y)=\exp(Y')\in H\cap \exp(\fp_+)$, where $Y'\in \fp_+$ is in the higher order part of the filtration than $Z$.

Thus the above computations can be seen as an induction step. If $Y'\in \fh$, then $\exp(2Z)=\exp(Y')\exp(2Z+Y)$, and thus $Z\in \fh_+$. Clearly, the filtration component of $Y'$ grows by each step, and thus, we get $Y'=0\in \fh$ after $k$ steps for $|k|$--graded geometry and the claim follows by induction.

Then consider for $h\in \Sigma_{\id_{\fg^{-1}/\fp}}^\fk$ the extension $(\alpha,\bar i)$ giving the simply connected covering from proof of Lemma \ref{lemma3.2}. Thus $h\in \exp(\fh_+)$ as above and $\Sigma_{\id_{\fg^{-1}/\fp}}^\fk= \exp(\fh_+)$.
\end{proof}

There is the following crucial consequence of the above Jordan decomposition from the Lemma \ref{h-minus}, which in particular simplifies the computation of $\sinf_J$.

\begin{Theorem}\label{2.5}
Let $(\pi: \ba\to M,\om)$ be a homogeneous parabolic geometry of type $(G,P)$. Then there is $u_0\in \ba$ such that $$\Sigma_J^\fk=\{s_J \exp Z : Z\in \fh_+\}$$
for the extension $(i,\alpha)$ of effective $(K,H)$ to $(G,P)$ giving the geometry at $u_0$ with $\Sigma_J^\fk\neq \emptyset$.

In particular, if $\sinf_J\neq\emptyset $, then there is $u_0\in \ba$ such that $$\sinf_J=\{s_J \exp Z : Z\in (\frak a_{u_0})_+\}$$
and either $A_{u_0}\cap \Sigma_J=\emptyset$ or $A_{u_0}\cap \Sigma_J=\sinf_J$ holds.
\end{Theorem}
\begin{proof}
We know from the Lemma \ref{h-minus} that there is $u_0$ such that if $\Sigma_J^\fk\neq \emptyset$, then $s_J\in \Sigma_J^\fk$ holds. Since $s_J^{-1}\cdot \Sigma_J^\fk=\Sigma_{\id_{\fg^{-1}/\fp}}^\fk$, the first claim follows from the Lemma \ref{h-plus}. In particular, the second claim follows for the case $\fk=\infi(u_0)$ and if $A_{u_0}\cap \Sigma_J\neq \emptyset$, then $s_J\in A_{u_0}$ and the last claim follows.
\end{proof}

\section{Homogeneous parabolic geometries and generalized symmetries}

{\bf Characterization of generalized symmetries.}
Now, we are ready to start the investigation of the generalized symmetries of a particular type $\Sigma(s_J)$ for $s_J\in Z(G_0)$. In particular, our results about automorphisms of type $\Sigma_J$ provide several equivalent characterizations of homogeneous parabolic geometries admitting $s_J$--symmetries.

\begin{Theorem}\label{homsym-central}
Let $(\pi: \ba\to M,\om)$ be a homogeneous parabolic geometry of type $(G,P)$ and $s_J\in Z(G_0)$. Then there is $u_0\in \ba$ such that the following claims are equivalent:

\begin{enumerate}
\item $A_{u_0}\cap \Sigma_J\neq \emptyset$, i.e., the geometry admits automorphisms of type $\Sigma_J$,
\item $s_J\in A_{u_0}$, i.e., the geometry admits $s_J$--symmetries.
\end{enumerate}
Moreover, there is $u_0\in \ba$ such that the following claims are equivalent:
\begin{enumerate}
\item $\sinf_J\neq\emptyset$.
\item there is local $s_J$--symmetry preserving $\frak{inf}(u_0)$.
\item $\frak{inf}(u_0)$ splits into eigenspaces of $\Ad_{s_J}$ in $\fg$ and the components of the harmonic curvature such that $\lambda_J(\mu)\neq 0 \mod 2\pi i$ for the corresponding weight $\mu$ vanish at $u_0$.
\end{enumerate}
\end{Theorem}
\begin{proof}
The first equivalence is consequence of the Theorem \ref{2.5}. In the case of the second equivalence, the description of $\sinf_J$ at $u_0\in \ba$ in the Theorem \ref{2.5} shows that (1) and (2) are equivalent. Further, claims (1) and (2) of the Lemma \ref{altcon} show for the choice $\fk=\infi(u_0)$ that (2) is equivalent with (3).
\end{proof}

We proved in the Theorem \ref{2.5} that the automorphisms of type $\Sigma_J$ are in bijection with elements of $(\fa_{u_0})_+$. The following Corollary of the Lemma \ref{uscor} shows which elements of $(\fa_{u_0})_+$ correspond to generalized symmetries of type $\Sigma(s_J)$.

\begin{Corollary} \label{4.2}
Let $(\pi: \ba\to M,\om)$ be a homogeneous parabolic geometry of type $(G,P)$ such that $s_J\in A_{u_0}$. Then $s_J$--symmetries are in bijection with the elements of the set $$(\fa_{u_0})_+\cap \{C(-\Ad_{s_J}^{-1}(Y),Y) : Y\in \fp_+\},$$ where the operator $C$ is given by the BCH--formula.
\end{Corollary}

\noindent
{\bf Generalized symmetries and harmonic curvature.}
Let us now describe how the remaining curvature restricts $\sinf_J$ and vice versa. There is the following action of $\frak a_{u_0}$ on the harmonic curvature 
for each one--parameter group of automorphisms $\phi_t$ generated by $X\in \frak a_{u_0}$:
$$0={d \over dt}|_{t=0}\kappa_H(u_0)={d \over dt}|_{t=0}\kappa_H(\phi_t(u_0))=X.\kappa_H(u_0).$$ 
This action together with the results in \cite{KT} allows us to formulate the following Theorem describing the relations between $(\frak{a}_{u_0})_+$ and $\kappa_H(u_0)$.

\begin{Theorem}\label{curvrest}
Let $(\pi:\ba\to M,\om)$ be a homogeneous parabolic geometry of type $(G,P)$ and assume $(\frak{a}_{u_0})_+\neq \emptyset$.  Then the following facts hold:
\begin{enumerate}
\item Denote $$\frak a_0(\kappa_H(u_0)):=\{X\in \fg_0: X.\kappa_H(u_0)=0\},$$ $$\frak a_i(\kappa_H(u_0)):=\{X\in \fg_i : [X,\fg_{-1}]\subset \frak a_{i-1}(\kappa_H(u_0))\}.$$
Then $\gr(\frak{a}_{u_0}/(\frak{a}_{u_0})_+)\subset \frak a_0(\kappa_H(u_0))$ and $\gr(\frak{a}_{u_0}^i/\frak{a}_{u_0}^{i+1})\subset \frak a_i(\kappa_H(u_0))$, where $\gr$ is the chosen identification of $\fg^i/\fg^{i+1}$ with $\fg_i$.
\item For an arbitrary component $\mu$ of the harmonic curvature in $H^2(\fg_-,\fg)$ with non--trivial intersection with $\kappa_H(u_0)$ holds: $$dim(\gr(\frak{a}_{u_0}^i/\frak{a}_{u_0}^{i+1}))\leq dim(\frak a_i(\kappa_H(u_0)))\leq dim(\frak a_i(X^{-\mu})).$$
\item There is a simple restricted root $\gamma$ of a positive height such that $\langle \gamma,\tilde \mu_i \rangle=0$ for $\tilde \mu_i$ in $H^2(\fp_+,\fg)$ corresponding to non--trivial components of $\kappa_H(u_0)$.
\item Assume $\sinf_J\neq\emptyset$. Then the eigenspaces of $\Ad_{s_J}$ in $(\frak{a}_{u_0})_+$ can have eigenvalues only of the form $e^{\lambda_J(\alpha)}$ for restricted roots $\alpha$ spanned by roots $\gamma$ satisfying the claim (3).
\end{enumerate}
\end{Theorem}
\begin{proof}
The first step to prove all the claims is to complexify $\fg$, $\fp$, $\kappa_H$, $\mu$, $G_0$ and $(\frak{a}_{u_0})_+$. 

Since $(\frak{a}_{u_0})_+$ acts trivially on the harmonic curvature, we get $\gr(\frak{a}_{u_0}/(\frak{a}_{u_0})_+)\subset \frak a_0(\kappa_H(u_0))$. Since $\fg^{-1}/\fp\subset \frak{inf}(u_0)/\frak{a}_{u_0}$, we get $\gr(\frak{a}_{u_0}^i/\frak{a}_{u_0}^{i+1})\subset \frak a_i(\kappa_H(u_0))$ for all $i$ and we proved the claim (1).

Then we can apply the results from \cite{KT}. Namely, the Proposition 3.1.1 implies the claim (2), and the Theorem 3.3.3 implies the claim (3).

We will use the claim (2) of the Lemma \ref{altcon} to modify the Proof of \cite[Proposition 3.1.1]{KT} in order to prove the claim (4). We again consider the complexified situation and let $Z$ belong to $\gr(\frak{a}_{u_0}^i/\frak{a}_{u_0}^{i+1})\subset \frak a_i(\kappa_H(u_0))$. Then the components of $Z$ in different eigenspaces of $\Ad_{s_J}$ are in $\frak a_i(\kappa_H(u_0))$, too. Thus we only have to prove that the eigenspaces of $\Ad_{s_J}$ in $\frak a_i(\kappa_H(u_0))$ satisfy the claim.

According to the proof of \cite[Proposition 3.1.1]{KT}, for any weight $\mu'$ in $H^2(\fg_-,\fg)$ in a non--trivial component $\mu$ of $\kappa_H(u_0)$, there is a sequence $g_i$ of elements of complexified $G_0$ such that $g_i.\mu'$ converges to $-\mu$. Thus $g_i.\frak a_i(\kappa_H(u_0))$ converges to a subset of $\frak a_i(X^{-\mu})$, which is described by \cite[Theorem 3.3.3]{KT}. The consequence of the description of $\frak a_i(X^{-\mu})$ is that $\Ad_{s_J}$ has only the claimed eigenvalues on $\frak a_i(X^{-\mu})$.

We modify the proof of \cite[Proposition 3.1.1]{KT} to prove that the dimension of the  eigenspace of $\Ad_{s_J}$ in $\frak a_i(\kappa_H(u_0))$ is lower than the dimension of that eigenspace in $\frak a_i(X^{-\mu})$. Precisely, the consequence of claim (2) of the Lemma \ref{altcon} is that we can discuss each eigenspace of $\Ad_{s_J}$ separately, and the argumentation in the proof \cite[Proposition 3.1.1]{KT} is still valid. Thus the dimension of the eigenspace of $\Ad_{s_J}$ in $\frak a_i(g_i.X^\mu)$ does not drop, and the claim (4) follows.
\end{proof}

\begin{Remark}
There are tables presented in \cite{KT} giving the classification of possible roots $\gamma$ for components of the harmonic curvature of simple complex parabolic geometries. We will complete the list for the real cases of our interest in the last section.
\end{Remark}

The above Theorem has a simple but important consequence.

\begin{Corollary}\label{4.neco}
Let $(\pi:\ba\to M,\om)$ be a non--flat homogeneous parabolic geometry of type $(G,P)$, and assume that $\lambda_J(\gamma)=0  \mod 2\pi i$ for all simple restricted roots satisfying the claim (3) of the Theorem \ref{curvrest}. Then there is at most one $s_J$--symmetry at any point.
\end{Corollary}

\section{Important classes of generalized symmetries of homogeneous geometries}
\label{important}

{\bf Automorphisms with higher order fixed points.}
The first natural action on $T^{-1}_xM$ we can consider is simply the identity. In such case, the Theorem \ref{2.5} states that automorphisms of type $\Sigma_{\id}$ on a homogeneous parabolic geometry are precisely the automorphisms with higher order fixed points investigated in \cite{C1}. 
The existence of $s_\id$--symmetry provides no new informations about the parabolic geometry, because the only $s_\id$--symmetry is the identity on $M$. Thus the harmonic curvature can only be restricted by the number of automorphisms of type $\Sigma_{\id}$ as in Theorem \ref{curvrest}.

\noindent
{\bf Symmetries of parabolic geometries.}
The second most natural action on $T^{-1}_xM$ we can consider is $-\id$. We already investigated geometries with automorphisms of type $\Sigma_{-\id}$ in \cite{G2,G3,Z1,Z2} and called them symmetries in these articles. In this article, we will call them \emph{usual symmetries} to distinguish them from the other types of generalized symmetries. Moreover, for comparison with the other types of generalized symmetries, it is convenient to denote $J$ as tuple $(-,\dots,-)$ representing the action $-\id$ of $J$ on the root spaces of each simple restricted root in $\Xi$. It is simple to decide about the existence of $J$ for a given parabolic geometry according to the Proposition \ref{lambda-action}.

\begin{Proposition}
Let $\fg$ be a semisimple Lie algebra, and let $\fp$ be a parabolic subalgebra of $\fg$ corresponding to $\Xi=\{i_1,\dots, i_j\}$, which does not contain any simple ideal of $\fg$. Then there exists a type $(G,P)$ with $s_{(-,\dots,-)}\in Z(G_0)$.
\end{Proposition}

We can describe the eigenspaces of $\Ad_{s_{(-,\dots,-)}}$ in a fairly simple way as follows.

\begin{Proposition}\label{5.2}
Let $(\ba\to M, \om)$ be a parabolic geometry of type $(G,P)$ such that $s_{(-,\dots,-)}\in Z(G_0)$. Then the following facts hold:

\begin{enumerate} 
\item The $-1$--eigenspace of $\Ad_{s_J}$ in $\fg$ equals to $\sum_{i\ {\rm odd}} \fg_i$, and the $1$--eigenspace of $\Ad_{s_J}$ in $\fg$ equals to $\sum_{i\ {\rm even}} \fg_i$.
\item $$\fp^{s_{(-,\dots,-)}}_{fix}=\sum_{i>0 \  {\rm even}}\fg_i.$$
\item A usual symmetry is $s_{(-,\dots,-)}$--symmetry if and only if the symmetry is involutive. In particular, $$\Sigma(s_{(-,\dots,-)})=\{s_{(-,\dots,-)}\exp(Z) :  Z\in \sum_{i>0 \  {\rm odd}}\fg_i\}.$$
\end{enumerate} 
\end{Proposition} 
\begin{proof}
The first claim follows from a simple computation with $\lambda_{(-,\dots,-)}$, and then the second claim follows, too. 

If $s_{(-,\dots,-)}\exp(Z)s_{(-,\dots,-)}\exp(Z)=\id$ for $Z\in \fp_+$, then $\Ad_{s_{(-,\dots,-)}}(Z)=-Z$, i.e., $Z\in \sum_{i>0\  {\rm odd}}\fg_i$. Thus $$s_{(-,\dots,-)}\exp(Z)=\exp(-\frac12 Z)s_{(-,\dots,-)}\exp(\frac12 Z)\in \Sigma(s_{(-,\dots,-)})$$ is $s_{(-,\dots,-)}$--symmetry. Conversely, it is clear that $p^{-1}s_{(-,\dots,-)}p\in \Sigma(s_{(-,\dots,-)})$ is involutive. In particular, the description of $\Sigma(s_{(-,\dots,-)})$ follows.
\end{proof}

Thus the Theorem \ref{homsym-central} can be rephrased due to the previous results as follows. 

\begin{Theorem}\label{homsym-usual}
Let $(\ba \rightarrow M,\om)$ be a homogeneous parabolic geometry of type $(G,P)$ such that $s_J\in Z(G_0)$ for $J=(-,\dots,-)$. Then there is $u_0\in \ba$ such that the following facts are equivalent:

\begin{enumerate}
\item There is a usual symmetry at one (and thus at each) point of $M$.
\item There is an involutive usual symmetry at one (and thus at each) point of $M$.
\end{enumerate}
Moreover, there is $u_0\in \ba$ such that the following claims are equivalent:
\begin{enumerate}
\item $\sinf_J\neq\emptyset$.
\item there is local usual symmetry preserving $\frak{inf}(u_0)$.
\item The space $\frak{inf}(u_0)$ splits into $\sum_{i\ {\rm odd}} \fg_i$ and $\sum_{i\ {\rm even}} \fg_i$, and the non--zero components of the harmonic curvature have even homogeneity.
\end{enumerate}
\end{Theorem}

In particular, there is the following consequence of the condition on the harmonic curvature.
\begin{Corollary}
If a homogeneous parabolic geometry does not admit a harmonic curvature of even homogeneity, and there is a usual symmetry at one point of $M$, then the whole curvature vanishes.

The classification of parabolic geometries with $\fg$ simple that admit a harmonic curvature of even homogeneity can be found in the tables in the last Section.
\end{Corollary}

Let us point out some important consequences of the classification in the tables in the last Section.

\begin{Proposition} \label{4.5}
The following facts hold for homogeneous parabolic geometries:

\begin{enumerate}
\item There is more than one usual symmetry if and only if there is more than one involutive usual symmetry.

\item If there is no simple restricted root $\gamma$ of positive height such that $\langle \gamma,\tilde \mu_i \rangle=0$ for $\tilde \mu_i$ in $H^2(\fp_+,\fg)$ corresponding to non--trivial components of $\kappa_H(u_0)$, then the non--flat homogeneous parabolic geometry has at most one (involutive) usual symmetry at each point.

\item If $|k|\leq 2$ and $\fg$ is simple, then the non--flat homogeneous parabolic geometry has at most one (involutive) usual symmetry at each point.
\end{enumerate}
\end{Proposition}
\begin{proof}
It follows from the Proposition \ref{5.2} that there is only one involutive usual symmetry if and only if $(\fa_{u_0})_+\subset \fp^{s_{(-,\dots,-)}}_{fix}=\sum_{i>0\ {\rm even}} \fg_i$, but this is not possible for obvious reasons. Thus the claim (1) follows.

The claim (2) follows from the Theorem \ref{curvrest}. The claim (3) then follows from the tables in the last Section.
\end{proof}

In particular, this Proposition extends the known restrictions on the number of usual symmetries to all AHS--structures and parabolic contact structures.

The homogeneous parabolic geometries with a unique usual symmetry at each point can be immediately related to the classical symmetric and reflexion spaces, see \cite{L3}.

\begin{Theorem}
Let $(\ba\to M,\om)$ be a homogeneous parabolic geometry with a single (involutive) usual symmetry at one (and thus at each) point of $M$. Then $M$ with these symmetries has a structure of a homogeneous reflexion space, i.e., it is a correspondence space to a symmetric space.
\end{Theorem}
\begin{proof}
We know that $ps_{(-,\dots,-)}p^{-1}$ is an automorphisms of type $\Sigma_{(-,\dots,-)}$ for any $p\in A_{u_0}$. Then it is clear that $s_{(-,\dots,-)}\in Z(A_{u_0})$ due to the uniqueness of the symmetry and thus the result follows from \cite[Theorem 3.4.4]{G4}.
\end{proof}

\begin{Remark}
We recall that although the group generated by usual symmetries acts transitively on $M$, the whole $\Aut(\ba,\om)$ does not have to act effectively on the symmetric space from the previous Theorem.
\end{Remark}

Let us briefly discuss the semisimple case. For arbitrary semisimple Lie algebra $\fg$, we will write $\fg=\sum_{l=1}^{r} \fg^{(l)}$, where $\fg^{(l)}$ are simple factors of $\fg$, which are naturally $|k_l|$--graded. Then $\mu^\fg=\oplus_{l=1}^{r} \mu^{\fg^{(l)}}$, i.e., the component of the harmonic curvature $(i,j)$ decomposes into factors $(i,j)^l$. 
Thus the components of the following types can appear, where $1\leq l_1, l_2, l_3\leq j$ are different, $b=2\frac{\langle\alpha_j,\alpha_i\rangle}{\langle\alpha_i,\alpha_i\rangle}$ and $c$ is the  homogeneity of $X^{\alpha_i} \otimes X^{s_{\alpha_i}(-\mu^{\fg^{(l_3)}})}\in \fp_+\otimes \fg$:
\begin{center}
\begin{table}[H]
\caption{Possible types of components of the harmonic curvature 
for the semisimple Lie algebra $\fg$}
\label{tab1}
\begin{tabular}{|c|c|c|c|}
\hline
&$(i,j)^l$ & homogeneity & restrictions\\
\hline
(1)&$\fg_{\alpha_i}\subset \fg^{(l_1)}$, $\fg_{\alpha_j}\subset \fg^{(l_2)}$, $l=l_3$& $1$ & $k_{l_3}=1$\\
\hline
(2)&$\fg_{\alpha_i}\subset\fg^{(l_1)}$, $\fg_{\alpha_j}\subset \fg^{(l_1)}$, $l=l_3$& $1-b-k_{l_3}$ & $j\notin \Xi|_{\fg^{(l_1)}}$, $b\neq 0$, $k_{l_3}<1-b$\\
\hline
(3)&$\fg_{\alpha_i}\subset\fg^{(l_1)}$, $\fg_{\alpha_j}\subset \fg^{(l_1)}$, $l=l_3$& $2-b-k_{l_3}$ & $j\in \Xi|_{\fg^{(l_1)}}$, $k_{l_3}<2-b$\\
\hline
(4)&$\fg_{\alpha_i}\subset\fg^{(l_1)}$, $\fg_{\alpha_j}\subset \fg^{(l_2)}$, $l=l_1$& $1+c$ & $-1<c$\\
\hline
(5)&$\fg_{\alpha_i}\subset\fg^{(l_1)}$, $\fg_{\alpha_j}\subset \fg^{(l_1)}$, $l=l_1$&  & \\
\hline
\end{tabular}
\end{table}
\end{center}

There is the following trivial but crucial observation about usual symmetries in the semisimple case.

\begin{Lemma} \label{poznamka}
If $\fg$ is semisimple, then there generically is more than one (involutive) usual symmetry at a single point of $M$. It suffices to take a simple factor $\fg^{(l)}$ which does not contribute to the $\kappa_H$ according to the Table \ref{tab1}.
\end{Lemma}

So we will discuss only the usual symmetries, which are different after restricting to simple factors $\fg^{(l_i)}$ which contribute to $\kappa_H$. 
Let us present here facts about possible components of the harmonic curvature of homogeneous parabolic geometries with $\Sigma_{(-,\dots,-)}^{\infi(u_0)}\neq \emptyset$ according to the line in the Table \ref{tab1}.
\smallskip \\
(1) The component always vanishes due to its homogeneity.
\smallskip \\
(2) The component vanishes with the exception of the cases given in the following table (and their complexifications).
\begin{center}
\begin{tabular}{|c|c|c|}
\hline
$\fg^{(l_1)}$ & $\Xi|_{\fg^{(l_1)}}$ contains & restrictions \\
\hline
$\frak{g}_2(2)$&$\{1\}$ & $k_{l_3}=2$\\
$\frak{f}_4(4)$&$\{2\}$ & $k_{l_3}=1$\\
$\frak{so}(n,n+1)$&$\{n\}$ & $k_{l_3}=1$\\
$\frak{sp}(2n,\mathbb{R})$&$\{n-1\}$ & $k_{l_3}=1$\\
\hline
\end{tabular}
\end{center}
In particular, if $\Xi|_{\fg^{(l_3)}}$ contains $\ell$ such that $\langle \mu^{\fg^{(l_3)}},\alpha_\ell\rangle=0$, then non--flat homogeneous parabolic geometries of the corresponding type admit more than one usual symmetry at each point.
\smallskip \\
(3) The component vanishes with the exception of the cases given in the following table (and their complexifications).
\begin{center}
\begin{tabular}{|c|c|c|}
\hline
$\fg^{(l_1)}$ & $\Xi|_{\fg^{(l_1)}}$ contains & restrictions \\
\hline
$\frak{g}_2(2)$&$\{1,2\}$ & $k_{l_3}=1,3$\\
$\frak{f}_4(4)$&$\{2,3\}$ & $k_{l_3}=2$\\
$\frak{so}(n,n+1)$&$\{n-1,n\}$ & $k_{l_3}=2$\\
$\frak{sp}(2n,\mathbb{R})$&$\{n-1,n\}$ & $k_{l_3}=2$\\
\hline
arbitrary&$\{i,j\}$ & $b=1, k_{l_3}=1$\\
\hline
\end{tabular}
\end{center}
In particular, if $\Xi|_{\fg^{(l_3)}}$ contains $\ell$ such that $\langle \mu^{\fg^{(l_3)}},\alpha_\ell \rangle=0$, then non--flat  homogeneous parabolic geometries of the corresponding type admit more than one usual symmetry at each point.
\smallskip \\
(4) The component vanishes with the exception of the cases satisfying $c=1$, which are exactly projective and contact projective types. Then, if there is $\ell\in \Xi|_{\fg^{(l_2)}}$ such that $\sigma^*(\alpha_\ell)\neq \pm \alpha_j$ and $\langle \alpha_\ell,\alpha_j\rangle=0$, then non--flat  homogeneous parabolic geometries of the corresponding type admit more than one usual symmetry at each point.
\smallskip \\
(5) The component is an element of $H^2(\fg^{(l_1)}_-,\mu^{\fg^{(l_1)}})$, and  results from the simple case apply.

\noindent
{\bf Generalized symmetries of order $2$ and parabolic geometries with weak para--complex structures.}
Let us now consider automorphisms of type $\Sigma_J$, where $J$ is a natural weak para--complex structure on $T^{-1}_xM$, i.e., $J^2=\id$ holds. We will show that there can be many different para--complex structures on $T^{-1}_xM$, and it is convenient to denote $J$ as tuple with entries $+$ or $-$ that indicate, whether the action of $J$ is
$\id$ or $-\id$ on root spaces of each simple restricted root in $\Xi$.
The possible para--complex structures $J$ for a given parabolic geometry are determined by the Proposition \ref{lambda-action}.

\begin{Proposition}
Let $\fg$ be a semisimple Lie algebra, and let $\fp$ be a parabolic subalgebra of $\fg$ corresponding to $\Xi=\{i_1,\dots, i_j\}$, which does not contain any simple ideal of $\fg$. Then there exists a type $(G,P)$ with $s_J\in Z(G_0)$ such that the corresponding $J$ is a weak almost para--complex structure on $T^{-1}M$ if and only if $j>2$, or $j=2$ and $\sigma^*(\alpha_{i_1})\neq\alpha_{i_2}$. 

All possible para--complex structures $J$ correspond to decompositions of $\Xi$ into $\Xi^-=\{i_\ell\in \Xi: J_\ell=-\}$ and $\Xi^+=\{i_\ell\in \Xi: J_\ell=+\}$, which are stable under the complex conjugation $\sigma^*$.
\end{Proposition}

The structure of eigenspaces of $\Ad_{s_J}$ is more complicated in this case. We have to consider the bi--grading $\fg_{(a,b)}\subset \fg_{a+b}$ of $\fg$ with respect to $\Xi^-$ and $\Xi^+$. Then we can characterize the $P$--conjugacy class $\Sigma(s_{J})$ as follows.

\begin{Proposition}
Let $(\ba\to M, \om)$ be a parabolic geometry of type $(G,P)$ such that $s_{J}\in Z(G_0)$ for the para--complex structure $J$. Then the following facts hold:

\begin{enumerate} 
\item The $-1$--eigenspace of $\Ad_{s_J}$ in $\fg$ is $\sum_{a\ {\rm odd}}\fg_{(a,b)}$, and the $1$--eigenspace of $\Ad_{s_J}$ in $\fg$ is $\sum_{a\ {\rm even}}\fg_{(a,b)}$.
\item $$\fp^{s_{J}}_{fix}=\sum_{a+b>0,\  a\ {\rm even}}\fg_{(a,b)}.$$
\item An automorphism of type $\Sigma_J$ is $s_{J}$--symmetry if and only if the symmetry is involutive. In particular, 
$$\Sigma(s_{J})=\{s_{J}\exp(Z) : Z\in \sum_{a+b>0,\  a\ {\rm odd}}\fg_{(a,b)}\}.$$
\end{enumerate} 
\end{Proposition} 
\begin{proof}
The first claim follows from a simple computation with $\lambda_{J}$ on $\fg_{(a,b)}$, and then the second claim is clear. 

If $s_{J}\exp(Z)s_{J}\exp(Z)=\id$ for $Z\in \fp_+$, then $\Ad_{s_{J}}(Z)=-Z$ holds, i.e., $Z\in \sum_{a+b>0,\  a\ {\rm odd}}\fg_{(a,b)}$. Thus 
$$s_{J}\exp(Z)=\exp(-\frac12 Z)s_{J}\exp(\frac12 Z)\in \Sigma(s_{J})$$ 
is $s_{J}$--symmetry. Conversely, it is clear that $p^{-1}s_{J}p\in \Sigma(s_{J})$ is involutive. In particular, the description of $\Sigma(s_{J})$ follows.
\end{proof}

The condition (3) of Theorem \ref{homsym-central} implies the following.

\begin{Proposition}
Let $(\ba \rightarrow M,\om)$ be a homogeneous parabolic geometry of type $(G,P)$ such that $\sinf_J\neq\emptyset$ for the para--complex structure $J$. Then the bi--homogeneity $(a,b)$ of each component  $\kappa_H$ has $a$ even.

The classification of parabolic geometries with $\fg$ simple that admit non--trivial harmonic curvatures can be found in the tables in the last Section.
\end{Proposition}

The statement analogous to the Proposition \ref{4.5} is more complicated, because we have to distinguish between the number of automorphisms of type $\Sigma_J$ and the number of $s_J$--symmetries.

\begin{Proposition}
The following facts hold for non--flat homogeneous parabolic geometries:

\begin{enumerate}
\item If there is no simple restricted root $\gamma$ of positive height such that $\langle \gamma,\tilde \mu_i \rangle=0$ for $\tilde \mu_i$ in $H^2(\fp_+,\fg)$ corresponding to non--trivial components of $\kappa_H(u_0)$, then there is at most one automorphism of type $\Sigma_J$ at each point.

\item If there is no simple restricted root $\gamma$ such that $\fg_{\gamma}\in \fg_{(1,0)}$ and $\langle \gamma,\tilde \mu_i \rangle=0$ for $\tilde \mu_i$ in $H^2(\fp_+,\fg)$ corresponding to non--trivial components of $\kappa_H(u_0)$, then there is at most one $s_J$--symmetry at each point.
\end{enumerate}
\end{Proposition}
\begin{proof}
The claim (1) follows from the Theorem \ref{curvrest}, and the claim (2) is a simple consequence of the Corollary \ref{4.neco}.
\end{proof}

 Let us point out the important consequence of the Theorem \ref{curvrest}.

\begin{Theorem}
Let $(\ba\to M,\om)$ be a homogeneous parabolic geometry with a single $s_J$--symmetry at one (and thus at each) point of $M$. Then $M$ with these $s_J$--symmetries has a structure of a homogeneous reflexion space, i.e., it is a correspondence space to a non--effective homogeneous symmetric space.
\end{Theorem}
\begin{proof}
We know that $ps_{(-,\dots,-)}p^{-1}$ is an automorphisms of type $\Sigma_{(-,\dots,-)}$ for any $p\in A_{u_0}$. Then it is clear that $s_{(-,\dots,-)}\in Z(A_{u_0})$ due to the uniqueness of the $s_J$--symmetry and thus the result follows from \cite[Theorem 3.4.4]{G4}. However, since the group generated by $s_J$--symmetries does not have to act transitively on $M$, then $\Aut(\ba,\om)$ does not have to act effectively on the symmetric space.
\end{proof}

Since the $1$--eigenspace of $\Ad_{s_J}$ is too big, there are to many possible cases we have to distinguish, when $\fg$ is semisimple. So although it is possible to derive the analogous restrictions as in the case of usual symmetries, we will not discuss them here explicitly. We only recall that the Lemma \ref{poznamka} is still valid for this type of generalized symmetries.

\noindent
{\bf Generalized symmetries of order $4$ and parabolic geometries with complex structures.}
Finally, we consider automorphisms of type $\Sigma_J$, where $J$ is a natural complex structure on $T^{-1}_xM$, i.e., $J^2=-\id$ holds. We will show that there can be many different complex structures on $T^{-1}_xM$, and it is convenient to denote $J$ as tuple with entries $i$ or $-i$  that indicate, whether the action of $J$ is $i\cdot \id$ or $-i\cdot \id$ on the root space of each simple restricted root in $\Xi$. The possible complex structures $J$ for a given parabolic geometry are determined by the Proposition \ref{lambda-action} as follows.

\begin{Proposition}
Let $\fg$ be a semisimple Lie algebra, and let $\fp$ be a parabolic subalgebra of $\fg$ corresponding to $\Xi=\{i_1,\dots, i_j\}$, which does not contain any simple ideal of $\fg$. Then there exists a type $(G,P)$ with $s_J\in Z(G_0)$ such that $J$ is an almost complex structure on $T^{-1}M$ if and only if $\sigma^*(\alpha_{i_l})\neq -\alpha_{i_l}$ for $l=1, \dots, j$. 

All possible complex structures $J$ correspond to decompositions of $\Xi$ into $\Xi^-=\{i_l\in \Xi: J_l=-i\}$ and $\Xi^+=\{i_l\in \Xi: J_l=i\}$ such that $\sigma^*(\Xi^-)=\Xi^+$.
\end{Proposition}

These generalized symmetries can be naturally related to the previous two classes of generalized symmetries.

\begin{Corollary}
Let $(\ba\to M,\om)$ be a homogeneous parabolic geometry with $s_J$--symmetry at one (and thus at each) point of $M$. Then there is $s_J^2=s_{(-,\dots,-)}$--symmetry at one (and thus at each) point of $M$. In general, the composition $J\circ J'$ of two complex structures $J$ and $J'$ is a weak para--complex structure.
\end{Corollary}

In other words, if there is at least one complex structure, then we can use the results from the previous two cases to get all of them.

There again is the bi--grading $\fg_{(a,b)}\subset \fg_{a+b}$ of $\fg$ with respect to $\Xi^-$ and $\Xi^+$. The eigenspaces of $\Ad_{s_J}$ can be described as follows.

\begin{Proposition}
Let $(\ba\to M, \om)$ be a parabolic geometry of type $(G,P)$ such that $s_{J}\in Z(G_0)$ for the complex structure $J$. Then the following facts hold:

\begin{enumerate} 
\item The $i^c$--eigenspace of $\Ad_{s_J}$ in $\fg$ is $\sum_{3a+b=c \mod 4}\fg_{(a,b)}$. In particular, the $1$--eigenspace of $\Ad_{s_J}$ in $\fg$ is $\sum_{3a+b=0 \mod 4}\fg_{(a,b)}$.
\item $$\fp^{s_{J}}_{fix}=\sum_{a+b>0,\  3a+b=0 \mod 4}\fg_{(a,b)}.$$
\end{enumerate} 
\end{Proposition} 
\begin{proof}
The first claim follows from a simple computation with $\lambda_{J}$ on $\fg_{(a,b)}$, and then the second claim is clear. 
\end{proof}

The condition (3) of the Theorem \ref{homsym-central}  implies the following.

\begin{Proposition}
Let $(\ba \rightarrow M,\om)$ be a homogeneous parabolic geometry of type $(G,P)$ such that $\sinf_J\neq\emptyset$ for the complex structure $J$. Then the the bi--homogeneity $(a,b)$ of of each component of $\kappa_H$ satisfies $3a+b=0 \mod 4$.

The classification of parabolic geometries with $\fg$ simple that admit such harmonic curvatures can be found in the tables in the last Section.
\end{Proposition}
\begin{proof}
Since $(-i)^a(i)^b=1$ if and only if $3a+b=0 \mod 4$, the first claim is clear. 

In order to get the results in the tables, we discuss the condition on the  bi--homogeneity in more detail. Since the homogeneity has to be even, the possible homogeneities are $2$ with $(a,b)=(3,-1),(1,1),(-1,3)$, and $4$, which appear only in special dimensions and can be computed directly.

If $\fg$ is simple, but the complexification $\fg_\mathbb{C}$ is not simple, then $\fg_\mathbb{C}=\fg\oplus \fg$ with the highest root $\mu^{\fg_\mathbb{C}}=\mu^{\fg}+(\mu^\fg)'$, where we denote by $\alpha'$ the roots of the other copy of $\fg$. In particular, $\sigma^*(\alpha)=\alpha'$. So $J$ can be chosen arbitrarily on each root $\alpha\in \Xi$, and the values on $\alpha'\in \Xi$ are determined by the above choice. Then, according to the line in the Table \ref{tab1}, we have the following facts for corresponding components of the harmonic curvature of homogeneous parabolic geometries with $s_J\in A_{u_0}$:
\smallskip \\
(1) These components cannot appear. 
\smallskip \\
(2) These components cannot appear, because $b\leq k$.
\smallskip \\
(3) These components cannot appear, because $b< k$.
\smallskip \\
(4) These components have to have homogeneity $2$, i.e., they arise only on complex projective and complex contact projective geometries.
\smallskip \\
(5) There are two possibilities for the homogeneity of the component $(i,j)$. If $\alpha_j\notin \Xi$, then the homogeneity has to be $4$. If $\alpha_j\in \Xi$, then the homogeneity can be $2$ or $4$, and we simply check the bi--homogeneity for each of them.

If $\fg_\mathbb{C}$ is simple, we have to distinguish if $\sigma^*(\alpha_i)= \alpha_j$ or not for the component of the harmonic curvature $(i,j)$. If not, then the possibilities are the same as in (5). Otherwise, the bi--homogeneity has to be $(1,1)$ or $(2,2)$, which is possible only in the $\fg=\frak{su}(p,q)$--case.
\end{proof}

The Proposition \ref{4.5} naturally holds in this case, too, and we get the following result in the same way as for usual symmetries.

\begin{Theorem}
Let $(\ba\to M,\om)$ be a homogeneous parabolic geometry with a single $s_J$--symmetry at one (and thus at each) point of $M$. Then $M$ with these $(s_J)^2$--symmetries has a structure of a homogeneous reflexion space, i.e., it is a correspondence space to a complex symmetric space. This is particularly true, when $\fg$ is simple, $|k|\leq 2$ and the geometry is non--flat.
\end{Theorem}

Again, there are too many choices involved for the action of $J$, when $\fg$ is semisimple. So although we can derive the analogous restrictions as in the case of usual symmetries, we will not discuss them here. We only recall that the Lemma \ref{poznamka} is still valid for this type of generalized symmetries.

\section{Two general constructions of examples}

{\bf Extensions of correspondence spaces over symmetric spaces.}
Let us recall the construction of symmetric parabolic geometries from \cite{G2,G4}, which particularly allows to construct symmetries of all types discussed in the Section \ref{important}. The construction can be summarized in the following steps:

\begin{enumerate}
\item We take the effective pair $(K,H)$ such that there is $h\in H$ such that $H$ is contained in the centralizer of $h$ in $K$, and $K/H$ is connected.
\item We find the extension $(i,\alpha)$ of $(K,H)$ to $(G,P)$ such that $i(h)\in \Sigma(s_J)$ for some $s_J\in Z(G_0)$.
\item Then $\mathcal{F}_{\alpha}(K\to K/H,\om_K)$ will be $s_J$--symmetric, where $J=\underline{\Ad}_{i(h)}|_{\fg^{-1}/\fp}$.
\end{enumerate}

If $h^2=\id$, then $K/H\to K/Z_K(h)$ is a correspondence space over the symmetric space $K/Z_K(h)$, where $Z_K(h)$ is the centralizer of $h$ in $K$, and the corresponding generalized symmetries cover the symmetries on $K/Z_K(h)$. In particular, one can start with the known classifications of symmetric spaces and homogeneous bundles over them.

However, without any distinguished properties of the elements $h\in H$, there is no classification of the homogeneous spaces $K/H$, and we cannot use the construction explicitly.

There are many examples of non--flat homogeneous parabolic geometries constructed by the construction from symmetric spaces in \cite{G2,G4}. We present one family of them here.

\begin{Example}[Extension $\alpha: \frak{so}(3)\to\frak{su}(2,1)$]\label{exmpl}
Consider the one--parameter class of linear maps $\alpha_t: \frak{so}(3)\to\frak{su}(2,1)$ for $t\geq 1$ given by
\[\alpha_t(x_1,x_2,x_3)=
\left( \begin{array}{ccc}
\frac{1+t^4}{8t^2}ix_3 & * & \frac{-(15t^8-34t^4+15)}{128t^4}ix_3 \\
t x_1+\frac{i}{t}x_2 & -\frac{1+t^4}{4t^2}ix_3  & \frac{-3t^4+5}{16t}x_1+\frac{5t^4-3}{16t^3}ix_2 \\
2ix_3 & * & \frac{1+t^4}{8t^2}ix_3 \end{array} \right),\] 
where the entries denoted by $*$ are determined by the structure of $\frak{su}(2,1)$, and $(x_1,x_2,x_3)$ represents the element $$\left( \begin{array}{ccc}
0 & x_3& -x_1 \\
-x_3 & 0  & -x_2 \\
x_1 & x_2 & 0 \end{array} \right)\in \frak{so}(3).$$  The curvature of the  homogeneous parabolic geometry given at $u_0$ by an extension with such $\alpha_t$ is of the form
\[ \kappa((x_1,x_2,x_3),(y_1,y_2,y_3))(u_0)=\]
\[
\left( \begin{array}{ccc}
0 & * & 0 \\
0 & 0  & \frac{3(1-t^8)}{16t^5}(y_3x_2-x_3y_2)+\frac{3(1-t^8)}{16t^3}i(y_3x_1-x_3y_1)\\
0 & 0 & 0 \end{array} \right),\]
where the entries denoted by $*$ are determined by the structure of $\frak{su}(2,1)$, and $\kappa$ vanishes for $t=1$.

Thus the map $\alpha_1$ is a Lie algebra homomorphism and thus determines the Lie group homomorphism $\iota_1: K:=Spin(3)\to G:=PSU(2,1)$. Let $H$ be the preimage of $\iota_1(Spin(3))\cap P_{1,2}$. Then $(\iota_1|_H,\alpha_t)$ for arbitrary $t\geq 1$ is an extension of $(Spin(3),H)$ to $(PSU(2,1),P_{1,2})$ giving a regular normal parabolic geometry with the above curvature. We recall that $(PSU(2,1),P_{1,2})$ is the type of CR--geometries.

It follows from the construction that $$s_{(-,-)}=\left( \begin{array}{ccc}
-1 & 0 & 0 \\
0 & 1  & 0\\
0 & 0 & -1\end{array} \right)$$
is the usual symmetry for the geometry given by any of the above extensions.

Further, we can consider the element $$s_{(i,-i)}:=\left( \begin{array}{ccc}
i & 0 & 0 \\
0 & 1  & 0\\
0 & 0 & i \end{array} \right)\in Z(G_0),$$
and it is easy to check that  $\Ad_{s_{(i,-i)}}\alpha_t(\frak{so}(3))\subset \alpha_t(\frak{so}(3))$ and $\Ad_{s_{(i,-i)}}\kappa(u_0)=\kappa(u_0)$ hold if and only if $t=1$. Thus there is $s_{(i,-i)}$--symmetry on the parabolic geometry of type $(PSU(2,1),P_{1,2})$ given by the extension $(\iota_1|_H,\alpha_1)$, which is flat.
\end{Example}

\noindent
{\bf Deformations of nilpotent groups by the lowest weight of harmonic curvature.}
We introduce the construction, which generalizes the construction from \cite{KT} to the real case. The construction can be divided into the following steps:

\begin{enumerate}
\item We start with $\fg_-$ of the chosen type $(G,P)$. (In general, we can start with the sum $\frak b_-$ of all negative root spaces.)
\item We choose a component of the harmonic curvature, which should not vanish. We know that the lowest weight vector $X^{-\mu}$ of the weight $\mu$ representing the component of the harmonic curvature can be seen as a map $X^{-\mu}: \fg_-\wedge \fg_-\to \fg$. In our examples, the map $X^{-\mu}$ has its image in $\fg_-$. (Generically, the image is in $\frak b_-$, and there are several rank two geometries such that the image is in $\fp_+$ (for example see Example \ref{exmpl}).)
\item We consider the Lie algebra $\fn$ coinciding with $\fg_-$ as a vector space together with the Lie bracket $[\ ,\ ]_{\fg_-}+X^{-\mu}$. The proof that $\fn$ is a Lie algebra can be found in \cite[Lemma 4.1.1]{KT}. The natural identification $\alpha_\fn: \fn\to \fg_-$ is an extension of $(\exp(\fn),\{e\})$ to $(G,P)$ with the curvature $-X^{-\mu}$. (In general, $\frak b_0:=\frak b_-\cap \fp$ is nilpotent, and $\alpha_\fn: \fn\oplus \frak b_0\to \frak b_-$ is an extension of $(\exp(\fn\oplus \frak b_0),\exp(\frak b_0))$ to $(G,P)$.)
\item We compute $(A_{u_0})_0$ following the Lemma \ref{lemma3.2}, and we compute $(\fa(-X^{-\mu}))_+$ following the Theorem \ref{curvrest}. 
\item  We describe $Z(G_0)\cap (A_{u_0})_0$ in order to find all types of generalized symmetries that the geometry admits. If $(\fa(-X^{-\mu}))_+=0$, then we end here. But if $(\fa(-X^{-\mu}))_+\neq 0$, the corresponding infinitesimal automorphisms are not complete on $\exp(\fn)$, and we need to modify our construction in the following way to get $(\fa_{u_0})_+=(\fa(-X^{-\mu}))_+$:

We consider the decomposition of $G/P$ into Schubert cells of the form $$G/P=\bigsqcup_{w\in W^\fp} \exp(\frak b_-)\tilde w.eP,$$ where $W^\fp$ is the Hasse graph of $\fp$ in $\fg$, and $\tilde w$ are the corresponding elements of $G$. Then we define the manifold $M$ as
$$M:=\bigsqcup_{w\in W^{\bar \fp}} \exp(\frak b_-)\tilde w.eP,$$
where $W^{\bar \fp}$ is Hasse graph of $\bar{\fp}$ in $\bar \fg\subset \fg$, which consists of connected components of Dynkin diagram containing only the simple restricted roots $\alpha_i$ such that $\langle \alpha_i,\tilde \mu \rangle = 0$ for the corresponding $\tilde \mu$ in $H^2(\fp_+,\fg)$ and $\bar \fp=\bar \fg\cap \fp$. The following Proposition shows that there is a non--flat homogeneous parabolic geometry on $M$ with the stabilizer $(A_{u_0})_0\exp((\fa(-X^{-\mu}))_+)$.
\end{enumerate}

\begin{Proposition}
There is a Lie group $K$ with the Lie algebra $\fn\oplus \frak a$ acting transitively on $M$, which is homotopically equivalent to a subgroup of $G$ with the Lie algebra $\fg_-\oplus \frak a$. Moreover, this determines an extension of $(K,H)$ to $(G,P)$, where $H=(A_{u_0})_0\rtimes \exp((\fa(-X^{-\mu}))_+)$ is the stabilizer of a point in $M$. The parabolic geometry over $M=K/H$ of type $(G,P)$ given at $u_0$ by this extension is regular and normal, and $K$ coincides with the automorphism group of the geometry.
\end{Proposition}
\begin{proof}
It follows from the results in \cite[Section 3.3]{KT} that $(\fa(-X^{-\mu}))_+=\bar \fp_+$, and thus, there is a subgroup $K'$ of $G$ with the Lie algebra $\fg_-\oplus \frak \fa(-X^{-\mu})$ acting transitively on $M$. If we look on the intersection of $K'$ with the Iwasawa decomposition of $G$, then it is clear that the group $K'$ can be written as $\exp(\frak b_-)\bar A \bar C$, where $\bar N\bar A \bar C$ is the Iwasawa decomposition of the subgroup of $K'$ with the Lie algebra $\bar \fg$, which is homotopically equivalent to $K'$.

Since $\exp(\frak b_-)$ is simply connected, we can modify its Lie group structure to $\exp(\fn+\frak b_0)$, and $K:=\exp(\fn+\frak b_0)\bar A \bar C$ is a Lie group with the Lie algebra $\fn\oplus \frak \fa(-X^{-\mu})$ acting transitively on $M$, which is homotopically equivalent to $K'$.

So we get an extension of $(K,H)$ to $(G,P)$, and since the only component of the  curvature is harmonic, the geometry on $K/H$ given by the extension is regular and normal. Since the action of any other element of $P$ on the curvature is non--trivial, $K$ coincides with the automorphism group of the geometry.
\end{proof}

Now, we will present several examples of the above construction. We will discuss the details of the construction only in the first example. The remaining examples are analogous, and we will summarize here only the results.

\begin{Example}[Type $(\frak{sl}(4,\mathbb{R}),\fp_{1,2,3})$ with curvature $(2,1)$]

\begin{enumerate}
\item The pair $(\fg,\fp)=(\frak{sl}(4,\mathbb{R}),\fp_{1,2,3})$ has $\fg_{-}$ of the form
\[
\fg_-=\fb_-=\begin{pmatrix}
0 & 0 & 0 & 0 \cr
x_1 & 0 & 0 & 0 \cr
x_2 & x_3 & 0 & 0 \cr
x_5 & x_6 & x_4 & 0 \cr
\end{pmatrix}.
\]

\item There are five possible components of the harmonic curvature, and we choose the component $(2,1)$.
\item The curvature of the geometry given by the modified Lie bracket of $\fn$ is of the form
\[\kappa((x_1,\dots,x_{6}),(x_1',\dots,x_{6}'))(u_0)=\begin{pmatrix}0 & 0 & 0 & 0\cr 0 & 0 & 0 & 0\cr 0 & 0 & 0 & 0\cr 0 & 0 & x_3\,x_2'-x_2\,x_3' & 0\end{pmatrix}.\]

\item We compute
\[(A_{u_0})_0=\begin{pmatrix}
x_7 & 0 & 0 & 0 \cr
0 & x_8 & 0 & 0 \cr
0 & 0 & 1 & 0 \cr
0 &0 & 0 & \frac{1}{x_7x_8}\cr
\end{pmatrix},\]

\[(\frak a_{u_0})_+=\begin{pmatrix}
0 & x_9 & 0 & 0 \cr
0 & 0 & 0 & 0 \cr
0 & 0 & 0 & 0 \cr
0 &0 & 0 & 0 \cr
\end{pmatrix}.\]

\item It is obvious that $Z(G_0)\cap (A_{u_0})_0=(A_{u_0})_0$, and $\fg_{-1}$ decomposes into the eigenspaces with eigenvalues $\frac{x_8}{x_7},\frac{1}{x_8},\frac{1}{x_7x_8}$ of possible generalized symmetries.
Thus there are generalized symmetries of all possible types according to the Table \ref{tab5}, namely:

$$s_{(-,-,-)}=\begin{pmatrix}
1 & 0 & 0 & 0 \cr
0 & -1 & 0 & 0 \cr
0 & 0 & 1 & 0 \cr
0 &0 & 0 & -1 \cr
\end{pmatrix},$$
$$s_{(-,+,-)}=\begin{pmatrix}
-1 & 0 & 0 & 0 \cr
0 & 1 & 0 & 0 \cr
0 & 0 & 1 & 0 \cr
0 &0 & 0 & -1 \cr
\end{pmatrix},$$
and $s_{(+,-,+)}=s_{(-,-,-)}\circ s_{(-,+,-)}$.

Since $(\frak a_{u_0})_+\neq 0$, there is infinitely many generalized symmetries of the first two types at each point, and there is infinitely many automorphisms of type $\Sigma_{(+,-,+)}$.

\begin{Remark}
In this example, we get $\fn=\mathbb{R}\oplus \frak{n}_{2,3,5}$, where $\mathbb{R}$ is spanned by the $x_1$--entry, and $\frak{n}_{2,3,5}$ is the negative part of the grading of the famous parabolic geometry $(G_2(2),P_1)$ corresponding to the generic distribution of the growth rate $(2,3,5)$.

It is easy to check that $K=Gl(2,\mathbb{R})\rtimes \exp(\frak{n}_{2,3,5})\subset G_2(2)$, and $M$ is homotopically equivalent to $S^1\times \mathbb{R}^5$.
\end{Remark}
\end{enumerate}
\end{Example}

\begin{Example}[Type $(\frak{so}(3,5),\fp_{1,2})$ with curvature $(1,2)$]
In this example, the data for the construction are the following:
\[
\fg_-=\begin{pmatrix}0 & 0 & 0 & 0 & 0 & 0 & 0 & 0\cr x_1 & 0 & 0 & 0 & 0 & 0 & 0 & 0\cr x_2 & x_3 & 0 & 0 & 0 & 0 & 0 & 0\cr 0 & -x_{4} & -x_{5} & 0 & -x_1 & -x_2 & -x_{10} & -x_{8}\cr x_{4} & 0 & -x_{6} & 0 & 0 & -x_3 & -x_7 & -x_9\cr x_{5} & x_{6} & 0 & 0 & 0 & 0 & 0 & 0\cr x_{10} & x_7 & 0 & 0 & 0 & 0 & 0 & 0\cr x_{8} & x_9& 0 & 0 & 0 & 0 & 0 & 0\end{pmatrix},
\]
$$
\kappa((x_1,\dots,x_{10}),(x_1',\dots,x_{10}'))(u_0)=  
$$
$$
\begin{pmatrix}0 & 0 & 0 & 0 & 0 & 0 & 0 & 0\cr 0 & 0 & 0 & 0 & 0 & 0 & 0 & 0\cr 0 & 0 & 0 & 0 & 0 & 0 & 0 & 0\cr 0 & 0 & 0 & 0 & 0 & 0 & 0 & 0\cr 0 & 0 & x_1x_2'-x_1'x_2 & 0 & 0 & 0 & 0 & 0\cr 0 & -x_1x_2'+x_1'x_2 & 0 & 0 & 0 & 0 & 0 & 0\cr 0 & 0 & 0 & 0 & 0 & 0 & 0 & 0\cr 0 & 0 & 0 & 0 & 0 & 0 & 0 & 0\end{pmatrix}, 
$$

\[(A_{u_0})_0=\begin{pmatrix}\pm x_{14}x_{15} & 0 & 0 & 0 & 0 & 0 & 0 & 0\cr 0 & x_{14} &0 & 0 & 0 & 0 & 0 & 0\cr 0 & 0 & x_{15} & 0 & 0 & 0 & 0 & 0\cr 0 & 0 & 0 & \frac{\pm 1}{x_{14}x_{15}} & 0 & 0 & 0 & 0\cr 0 & 0 & 0 & 0 & \frac{1}{x_{14}} & 0 & 0 & 0\cr 0 & 0 & \frac{-{x_{11}}^{2}-{x_{12}}^{2}}{2} & 0 & 0 & \frac{1}{x_{15}} & -x_{11} & -x_{12}\cr 0 & 0 & x_{11} & 0 & 0 & 0 & cos(x_{13}) & -sin(x_{13})\cr 0 & 0 & x_{12} & 0 & 0 &0 & sin(x_{13}) & cos(x_{13}) \end{pmatrix},\]

\[(\frak a_{u_0})_+=\begin{pmatrix}0 & 0 & 0 & 0 & 0 & 0 & 0 & 0\cr 0 & 0 & x_{16} & 0 & 0 & 0 & 0 & 0\cr 0 & 0 & 0 & 0 & 0 & 0 & 0 & 0\cr 0 & 0 & 0 & 0 & 0 & 0 & 0 & 0\cr 0 & 0 & 0 & 0 & 0 & 0 & 0 & 0\cr 0 & 0 & 0 & 0 & -x_{16} &0 &0 & 0\cr 0 & 0 &0 & 0 & 0 & 0 & 0 & 0\cr 0 & 0 & 0 & 0 & 0 & 0 & 0 & 0\end{pmatrix}.\]

It is easy to check that $Z(G_0)\cap (A_{u_0})_0$ is given by $x_{11}=x_{12}=x_{13}=0$ and $x_{15}=1$, and $\fg_{-1}$ decomposes into the eigenspaces with eigenvalues $\pm 1,x_{14}$ of possible generalized symmetries.
Thus there are generalized symmetries of all possible types according to the Table \ref{tab3}, namely:

$$s_{(-,-)}:=\begin{pmatrix}1 & 0 & 0 & 0 & 0 & 0 & 0 & 0\cr 0 & -1 & 0 & 0 & 0 & 0 & 0 & 0\cr 0 & 0 & 1 & 0 & 0 & 0 & 0 & 0\cr 0 & 0 & 0 & 1 & 0 & 0 & 0 & 0\cr 0 & 0 & 0 & 0 & -1 & 0 & 0 & 0\cr 0 &0 & 0 & 0 & 0 & 1& 0 & 0\cr 0 & 0 & 0 & 0 & 0 & 0 & 1 & 0\cr 0 & 0 & 0 & 0 & 0 & 0 & 0 & 1\end{pmatrix},$$
$$s_{(+,-)}:=\begin{pmatrix}-1 & 0 & 0 & 0 & 0 & 0 & 0 & 0\cr 0 & -1 & 0 & 0 & 0 & 0 & 0 & 0\cr 0 & 0 & 1 & 0 & 0 & 0 & 0 & 0\cr 0 & 0 & 0 & -1 & 0 & 0 & 0 & 0\cr 0 & 0 & 0 & 0 & -1 & 0 & 0 & 0\cr 0 &0 & 0 & 0 & 0 & 1& 0 & 0\cr 0 & 0 & 0 & 0 & 0 & 0 & 1 & 0\cr 0 & 0 & 0 & 0 & 0 & 0 & 0 & 1\end{pmatrix},$$
and $s_{(-,+)}:=s_{(-,-)}\circ s_{(+,-)}$.

Since $(\frak a_{u_0})_+\neq 0$, there is infinitely many generalized symmetries of the first two types at any point, and there is infinitely many automorphisms of type $\Sigma_{(-,+)}$.

\begin{Remark}
The Lie group $K$ can be expressed in the following way: We decompose $\frak{n}$ as $\fn=\mathbb{R}\oplus \fn_{4,7,8,9}$, where $\mathbb{R}$ is spanned by the $x_3$--entry, and the part of $A_{u_0}$ not contained in $\bar G$ is the opposite parabolic subgroup $(P_1)^{op}$ in $PSO(1,3)$. Then we can write $K=PGl(2,\mathbb{R})\rtimes(\exp(\fn_{4,7,8,9})\ltimes (P_1)^{op})$.
\end{Remark}
\end{Example}

\begin{Example}[Type $(\frak{sp}(8,\mathbb{R}),\fp_{1,2,4})$ with curvature $(2,1)$]
In this example, the data for the construction are the following:
\[
\fg_-=\begin{pmatrix}0 & 0 & 0 & 0 & 0 & 0 & 0 & 0\cr x_1 & 0 & 0 & 0 & 0 & 0 & 0 & 0\cr x_6 & x_2 & 0 & 0 & 0 & 0 & 0 & 0\cr x_7 & x_3 & 0 & 0 & 0 & 0 & 0 & 0\cr x_{15} & x_{14} & x_{12} & x_{10} & 0 &-x_1 &-x_6 &-x_7\cr x_{14} & x_{13} & x_{11}& x_8 &0 & 0 & -x_2 &-x_3\cr x_{12} &x_{11} &x_9 & x_4 & 0 & 0 & 0 & 0\cr x_{10} & x_8& x_4 & x_5 & 0 & 0 & 0 & 0\end{pmatrix},
\]

$$\kappa((x_1,\dots,x_{15}),(x_1',\dots,x_{15}'))=$$ $$\begin{pmatrix}0 & 0 & 0 & 0 & 0 & 0 & 0 & 0\cr 0 & 0 & 0 & 0 & 0 & 0 & 0 & 0\cr 0 & 0 & 0 & 0 & 0 & 0 & 0 & 0\cr 0 & 0 & 0 & 0 & 0 & 0 & 0 & 0\cr 0 & 0 &0 & 0 & 0 & 0 & 0 & 0\cr 0 & 0 & 0 & 0 & 0 & 0 & 0 & 0\cr 0 & 0 & x_6x_2'-x_2x_6' & 0 & 0 & 0 & 0 & 0\cr 0 & 0 & 0 & 0 & 0 & 0 & 0 & 0\end{pmatrix},$$  

\[(A_{u_0})_0=\begin{pmatrix}x_{16} & 0 & 0 & 0 & 0 & 0 & 0 & 0\cr 0 & \frac{x_{17}^4}{x_{16}} & 0 & 0 & 0 & 0 & 0 & 0\cr 0 & 0 & x_{17} & 0 & 0 & 0 & 0 & 0\cr 0 & 0 & x_{19} & x_{18} & 0 & 0 & 0 & 0\cr 0 & 0 &0 & 0 & \frac{1}{x_{16}} & 0 & 0 & 0\cr 0 & 0 & 0 & 0 & 0 & \frac{x_{16}}{x_{17}^4} & 0 & 0\cr 0 & 0 & 0 & 0 & 0 & 0 & \frac{1}{x_{17}} & \frac{-x_{19}}{x_{17}x_{18}}\cr 0 & 0 & 0 & 0 & 0 & 0 & 0 & \frac{1}{x_{18}}\end{pmatrix},\]

\[(\frak a_{u_0})_+=\begin{pmatrix}0 & x_{20} & 0 & 0 & 0 & 0 & 0 & 0\cr 0 & 0 & 0 & 0 & 0 & 0 & 0 & 0\cr 0 & 0 & 0 & 0 & 0 & 0 & 0 & 0\cr 0 & 0 & 0 & 0 & 0 & 0 & 0 & x_{21}\cr 0 & 0 &0 & 0 & 0 & 0 & 0 & 0\cr 0 & 0 & 0 & 0 & -x_{20} & 0 & 0 & 0\cr 0 & 0 & 0 & 0 & 0 & 0 &0 &0\cr 0 & 0 & 0 & 0 & 0 & 0 & 0 &0\end{pmatrix}.\]

It is easy to check that $Z(G_0)\cap (A_{u_0})_0$ is given by $x_{17}=x_{18}$ and $x_{19}=0$, and $\fg_{-1}$ decomposes into the eigenspaces with eigenvalues  $\frac{x_{17}^4}{x_{16}^2},\frac{x_{16}}{x_{17}^3},\frac{1}{x_{17}^2}$ of possible generalized symmetries.
Thus there is the generalized symmetry
$$s_{(+,-,+)}:=\begin{pmatrix}-1 & 0 & 0 & 0 & 0 & 0 & 0 & 0\cr 0 & -1 & 0 & 0 & 0 & 0 & 0 & 0\cr 0 & 0 & 1 & 0 & 0 & 0 & 0 & 0\cr 0 & 0 & 0 & 1 & 0 & 0 & 0 & 0\cr 0 & 0 &0 & 0 & -1 & 0 & 0 & 0\cr 0 & 0 & 0 & 0 & 0 & -1 & 0 & 0\cr 0 & 0 & 0 & 0 & 0 & 0 & 1 & 0\cr 0 & 0 & 0 & 0 & 0 & 0 & 0 &1\end{pmatrix}.$$

Since $(\frak a_{u_0})_+\neq 0$ is contained in $\fp_{fix}^{s_{(+,-,+)}}$, there is infinitely many automorphisms of type $\Sigma_{(+,-,+)}$.

But why we did not find any generalized symmetry of type $\Sigma_{(-,-,-)}$ although they are indicated in the Table \ref{tab5}?
The reason is that $\lambda_{(-,-,-)}$ induces an outer automorphism of $Sp(8,\mathbb{R})$. Thus we have to consider $G':=\{B\in Gl(8,\mathbb{R}): B\mathbb{J}B^T=\pm \mathbb{J}\}$, where $\mathbb{J}$ gives the Lie group $Sp(8,\mathbb{R})$. It is not hard to check that the short exact sequence $Sp(8,\mathbb{R})\to G'\to G'/Sp(8,\mathbb{R})=\mathbb{Z}_2$ splits according to the choice of the block matrix $\begin{pmatrix} -\id&0\cr 0&\id\end{pmatrix}\in G'/Sp(8,\mathbb{R})$. If we repeat the construction for the type $(G',P_{1,2,4}')$, then we get the following additional generalized symmetries:
$$s_{(-,-,-)}:=\begin{pmatrix}-1 & 0 & 0 & 0 & 0 & 0 & 0 & 0\cr 0 & 1 & 0 & 0 & 0 & 0 & 0 & 0\cr 0 & 0 & -1 & 0 & 0 & 0 & 0 & 0\cr 0 & 0 & 0 & -1 & 0 & 0 & 0 & 0\cr 0 & 0 & 0 & 0 & 1 & 0 & 0 & 0\cr 0 &0 & 0 & 0 & 0 & -1& 0 & 0\cr 0 & 0 & 0 & 0 & 0 & 0 & 1 & 0\cr 0 & 0 & 0 & 0 & 0 & 0 & 0 & 1\end{pmatrix},$$
and $s_{(-,+,-)}=s_{(-,-,-)}\circ s_{(+,-,+)}$. Since $(\frak a_{u_0})_+\neq 0$, there is infinitely many of them at any point.
\end{Example}

\begin{Example}[Type $(\frak{sl}(5,\mathbb{R}),\fp_{2,3})$ with curvature $(2,1)$ or $(3,2)$]
\begin{enumerate}
\item
In this example, we will discuss three cases for the pair  $(\frak{sl}(5,\mathbb{R}),\fp_{2,3})$ with
\[
\fg_-=\begin{pmatrix}0 & 0 & 0 & 0 & 0\cr 0 & 0 & 0 & 0 & 0\cr x_1 & x_2 & 0 & 0 & 0\cr x_4 & x_5 & x_3 & 0 & 0\cr x_7 & x_8 & x_6 & 0 & 0\end{pmatrix}.
\]
\item We consider the harmonic curvatures $(2,1)$, $(3,2)$ or both of them at once.
\item
The curvature in the case $\mu=(2,1)$ is given by
$$\kappa(u_0)((x_1,\dots,x_{8}),(x_1',\dots,x_{8}'))=\begin{pmatrix}0 & 0 & 0 & 0 & 0\cr 0 & 0 & 0 & 0 & 0\cr 0 & 0 & 0 & 0 & 0\cr 0 & 0 & 0 & 0 & 0\cr 0 & 0 & x_1\,x_2'-x_2\,x_1' & 0 & 0\end{pmatrix},$$ 
the curvature in the case $\mu=(3,2)$ is given by
$$\kappa(u_0)((x_1,\dots,x_{8}),(x_1',\dots,x_{8}'))=\begin{pmatrix}0 & 0 & 0 & 0 & 0\cr 0 & 0 & 0 & 0 & 0\cr 0 & 0 & 0 & 0 & 0\cr 0 & 0 & 0 & 0 & 0\cr x_3\,x_5'-x_5\,x_3' & 0 & 0 & 0 & 0\end{pmatrix},$$
and the curvature in the remaining case is the sum of the above.
\item
In the case $\mu=(2,1)$,  $A_{u_0}=(A_{u_0})_0$ and is given by

\[\begin{pmatrix}x_9 & x_{10} & 0 & 0 & 0\cr x_{11}  & x_{12}  & 0 & 0 & 0\cr 0 & 0 & x_{13}  & 0 & 0\cr 0 & 0 & 0 & \frac{1}{x_{13}^4 } & 0\cr 0 & 0 & 0 & x_{14} &  \frac{x_{13}^3}{x_9x_{12}-x_{10}x_{11}}\end{pmatrix}.\]

In the case $\mu=(3,2)$, we compute
\[(A_{u_0})_0=\begin{pmatrix}\frac{1}{x_{10}^3x_{11}^3x_{12}^3} & 0 & 0 & 0 & 0\cr x_9 & x_{10} & 0 & 0 & 0\cr 0 & 0 & x_{11} & 0 & 0\cr 0 & 0 & 0 &x_{10}^2x_{11}^2x_{12}^2 & 0\cr 0 & 0 & 0 & x_{13} & x_{12}\end{pmatrix},\]
and
\[(\frak a_{u_0})_+=\begin{pmatrix}0 & 0 & 0 & 0 & 0\cr 0 & 0 & x_{14} & 0 & 0\cr 0 & 0 & 0 & 0 & 0\cr 0 & 0 & 0 &0 & 0\cr 0 & 0 & 0 & 0& 0\end{pmatrix}.\]

Finally, in the case of both curvatures together, $A_{u_0}$ is the intersection of both previous results, i.e.,

\[A_{u_0}=(A_{u_0})_0=\begin{pmatrix}x_{11}^6 & 0 & 0 & 0 & 0\cr x_9 & x_{10} & 0 & 0 & 0\cr 0 & 0 & x_{11} & 0 & 0\cr 0 & 0 & 0 &\frac{1}{x_{11}^4} & 0\cr 0 & 0 & 0 & x_{12}& \frac{1}{x_{10}x_{11}^3} \end{pmatrix}.\]
\item
In the case $\mu=(2,1)$, elements of $Z(G_0)\cap A_{u_0}$ are of the form
\[s_{(a^{-5},a^{-10})}=\begin{pmatrix}a^7 & 0 & 0 & 0 & 0\cr 0 & a^7 & 0 & 0 & 0\cr 0 & 0 & a^2 & 0 & 0\cr 0 & 0 & 0 & a^{-8} & 0\cr 0 & 0 & 0 & 0 & a^{-8}\end{pmatrix}\]
for $a\in \mathbb{R}$, and $s_{(-,+)}$ for $a=-1$ is the unique $s_{(-,+)}$--symmetry. 

In the case $\mu=(3,2)$,  elements of $Z(G_0)\cap A_{u_0}$ are of the form
\[s_{(a^5,a^4)}=\begin{pmatrix}a^{-3} & 0 & 0 & 0 & 0\cr 0 & a^{-3}  & 0 & 0 & 0\cr 0 & 0 & a^{2}  & 0 & 0\cr 0 & 0 & 0 &  a^{2}  & 0\cr 0 & 0 & 0 & 0 & a^{2}  \end{pmatrix}\]
for $a\in \mathbb{R}$, and $s_{(-,+)}$ for $a=-1$ is $s_{(-,+)}$--symmetry and there is infinitely many of them. 

In the case of both curvatures, $Z(G_0)\cap A_{u_0}$ is generated by $s_{(-,+)}$, which is the unique $s_{(-,+)}$--symmetry.
\end{enumerate}
\end{Example}

\section{Tables}

We present the tables describing the possibilities for the existence of non--trivial homogeneous parabolic geometries with generalized symmetries discussed in the Section \ref{important}. We will present here the complete list of the cases with $\fg$ simple. Here:

\begin{itemize}
\item[$\fg$] ... simple Lie algebras and/or their real forms, if there are any restrictions on parameters, we will write them behind or below.
\item[$\Xi$] ... the set of roots, which gives the parabolic subalgebra $\fp$ of $\fg$, if there are any restrictions on parameters, we will write them below and we will omit the restrictions in the cases $\fg=\frak{sl}(n+1,\mathbb{H}), \frak{sp}(p,q)$, where the elements of $\Xi$ are even. 
\item[$\kappa_H$] ... components of the harmonic curvature of the parabolic geometry associated to the pair $(\fg,\fp)$, we indicate by $'$ that the corresponding root is in the other copy of $\fg$ in the complexification of complex Lie algebra $\fg$.
\item[$\gamma$] ... simple restricted roots that satisfy the condition (3) from the Theorem \ref{curvrest} for the given components of the harmonic curvature,  if there are any restrictions on parameters, we will write them behind or below.
\item[$J$] ... all the endomorphisms $J$ discussed in the Section \ref{important} such that $s_J$ acts trivially on the given component of the harmonic curvature, except the case that all complex structures $J$ can be obtained by the composition of the given complex structure with the remaining entries. If the both $+$ and $-$ possibilities are available, we will simply write $\pm$ on the corresponding position. 
\end{itemize}

The source of the data for the pairs $(\fg,\Xi)$ of simple parabolic geometries admitting a non--trivial curvature are the results in \cite{parabook,Y}. The data about the components of the harmonic curvature and roots $\gamma$ comes partly from the tables in \cite{KT}, and partly are computed by the algorithm presented in \cite{Si}. Finally, using the results from the Section \ref{important} and the explicit description of the  components of the harmonic curvature, we compute, which elements $s_J$ act trivially on the given components of the harmonic curvature. 

We order the tables by the size of $\Xi$ or number of non--trivial components of $\kappa_H$. We omit the cases which are equivalent to the other cases via outer automorphisms.

\newpage
\begin{table}[H]
\centering
\caption{\bf The case $|\Xi|=1$}
\label{tab2}
\begin{tabular}{|c|c|c|c|c|}
\hline
$\fg$ & $\Xi$ & $\kappa_H$& $\gamma$& $J$\\
\hline
$\frak{sl}(n+1,\{\mathbb{R, C}\})$,  $n>2$ & $\{1\}$ & $(1,2)$&& $(-)$ \\
$\frak{sl}(n+1,\mathbb{C})$,  $n>1$ & $\{1\}$ & $(1,1')$&& $(-),(i)$ \\
$\frak{sl}(n+1,\{\mathbb{R, C, H}\})$, $n>2$ &$\{2\}$ & $(2,1)$&& $(-)$ \\
\hline
$\frak{sl}(4,\{\mathbb{R,C,H}\})$ & $\{2\}$ & $(2,1)$, $(2,3)$ && $(-)$ \\
$\frak{sl}(n+1,\mathbb{C})$,  $n>2$& $\{1\}$ & $(1,2)$,$(1,1')$ & &$(-)$ \\
\hline
$\frak{so}(p,n-p)$,  $n>6, p> 0$ & $\{1\}$ & $(1,2)$&& $(-)$  \\
$\frak{so}(n,\mathbb{C})$, $n>6$ & $\{1\}$ & $(1,2)$&& $(-)$ \\
\hline
$\frak{sp}(2n,\{\mathbb{R,C}\})$, $n>2$ & $\{1\}$ & $(1,2)$&& $(-)$  \\
$\frak{sp}(2n,\mathbb{C})$, $n>1$ & $\{1\}$ & $(1,1')$ && $(-),(i)$ \\
$\frak{sp}(2n,\{\mathbb{R, C}\})$, $n>2$ & $\{2\}$ & $(2,1)$&& $(-)$  \\
$\frak{sp}(p,n-p)$, $n>2, p> 0$ & $\{2\}$ & $(2,1)$&& $(-)$ \\
\hline
$\frak{sp}(2n,\mathbb{C})$,  $n>2$ & $\{1\}$ & $(1,2)$, $(1,1')$ & &$(-)$ \\
\hline
$\frak{g}_{2}(2)$ & $\{1\}$ &$(1,2)$&& $(-)$\\
$\frak{g}_{2}(\mathbb{C})$ & $\{1\}$ &$(1,2)$&& $(-),(i)$\\
\hline
\end{tabular}
\end{table}

\begin{table}[H]
\caption{\bf The case $|\Xi|=2$, one component of $\kappa_H$}
\label{tab3}
\begin{tabular}{|c|c|c|c|c|}
\hline
$\fg$& $\Xi$ & $\kappa_H$& $\gamma$& $J$\\
\hline
$\frak{su}(1,2)$& $\{1,2\}$ & $(1,2)$&&$(-,-)$ \\
$\frak{sl}(n+1,\{\mathbb{R, C}\}), n>1$& $\{1,2\}$& $(1,2)$&&$(\pm,\pm)$\\
$\frak{sl}(n+1,\{\mathbb{R, C}\}), n>2$ & $\{1,2\}$ & $(2,1)$ &$\alpha_1$&$(+,-)$\\
$\frak{sl}(n+1,\{\mathbb{R, C}\}), n>2$ & $\{1,3\}$& $(1,2)$&&$(-,+)$\\
$\frak{sl}(n+1,\mathbb{R}), n>2$& $\{1,n\}$& $(1,n)$&&$(-,-)$\\
$\frak{sl}(n+1,\mathbb{C}), n>1$ & $\{p,p+1\}$ &  $ ((p+1)',p')$&$\alpha_p$&$(-,+)$ \\
$\frak{sl}(n+1,\mathbb{C}), n>1$ & $\{1,p\}$ &  $(1,p')$&&$(+,-)$\\
\hline
$\frak{su}(p+1,n-p)$& $\{1,n\}$ & $(1,n)$&&$(-,-)$,\ \\
$ n>2$&&&&$(i,-i)$\\
$\frak{sl}(n+1,\{\mathbb{R, C}\})$ & $\{1,p\}$& $(1,2)$&$\alpha_p$&$(-,+)$\\
$n>2$&$p>3$&&$n-1>p$&\\
$\frak{sl}(n+1,\{\mathbb{R, C}\})$& $\{1,p\}$ &  $(1,p)$&&$(+,-)$\\
$n>3$&$n>p>2$&&&\\
$\frak{sl}(n+1,\mathbb{C})$& $\{1,n\}$& $(1,n)$&&$(-,-)$,\\
$n>2$&&&&$(i,-i)$\\
$\frak{sl}(n+1,\{\mathbb{R, C, H}\})$ & $\{2,p\}$& $(2,1)$& $\alpha_p$&$(-,+)$\\
$n>2$&$p>2$&&$n>p>3$&\\
$\frak{sl}(n+1,\{\mathbb{R, C}\})$& $\{p,p+1\}$&  $ (p+1,p)$&$\alpha_p$& $(-,+)$ \\
$n>2$&$n-1>p>1$&&$n>p+1$&\\
$\frak{sl}(n+1,\mathbb{C})$ & $\{1,p\}$ &  $(1,1')$&$\alpha_p$&$(-,+)$\\
$n>1$&&&$n>p>2$&\\
\hline
$\frak{so}(3,4)$, $\frak{so}(7,\mathbb{C})$& $\{1,3\}$ & $(3,2)$&&$(-,-)$\\
$\frak{so}(p,n-p), n>4, p>2$& $\{1,2\}$ & $(1,2)$& $\alpha_2, n>6$& $(\pm,\pm)$\\
$\frak{so}(n,\mathbb{C}), n>4$& $\{1,2\}$ & $(1,2)$&  $\alpha_2, n>6$&$(\pm,\pm)$\\
$\frak{so}(p,n-p), n>6,p>1$ & \{1,2\} &$(2,1)$& &$(-,+)$\\
$\frak{so}(n,\mathbb{C}), n>6$ & \{1,2\} &$(2,1)$& &$(-,+)$\\
$\frak{so}(p,n-p), n>6,p>2$ & $\{2,3\}$& $(3,2)$&$\alpha_2, n>7$& $(-,+)$  \\
$\frak{so}(n,\mathbb{C}), n>6$ & $\{2,3\}$& $(3,2)$&$\alpha_2, n>7$& $(-,+)$  \\
$\frak{so}(n,n)$, $\frak{so}(2n,\mathbb{C}), n>3$& $\{1,n\}$& $(1,2)$&$\alpha_n, n>4$& $(-,+)$\\
\hline
$\frak{sp}(2n,\{\mathbb{R, C}\})), n>1$ & $\{1,2\}$& $(1,2)$& &$(-,+)$ \\
$\frak{sp}(2n,\{\mathbb{R, C}\})), n>2$  & $\{1,2\}$& $(2,1)$&$\alpha_1$& $(+,-)$\\
$\frak{sp}(2n,\{\mathbb{R, C}\})), n>2$  & $\{1,n\}$& $(1,2)$&$\alpha_n, n>3$& $(-,+)$\\
$\frak{sp}(2n,\{\mathbb{R, C}\}), n>2$   & $\{1,n\}$& $(1,n)$&& $(+,-)$\\
$\frak{sp}(2n,\{\mathbb{R, C}\})), n>2$  & $\{2,n\}$&  $(2,1)$&$\alpha_n, n>3$&$(-,+)$\\
$\frak{sp}(2n,\{\mathbb{R, C}\}), n>2$  & $\{n,n-1\}$&  $(n-1,n)$&&$(+,-)$\\
$\frak{sp}(n,n), n>1$  & $\{2,2n\}$&  $(2,1)$&$\alpha_{2n}$& $(-,+)$\\
$\frak{sp}(2n,\mathbb{C}), n>1$  & $\{1,n\}$ &   $(1,1')$&$\alpha_n$&$(-,+)$\\
$\frak{sp}(2n,\mathbb{C}), n>1$   & $\{1,n\}$& $(1,n')$&& $(+,-)$\\
$\frak{sp}(2n,\mathbb{C}), n>2$ & $\{n-1,n\}$  &  $((n-1)',n')$&&$(+,-)$\\
\hline
$\frak{g}_{2}(2)$ & $\{1,2\}$ &$(1,2)$ &&$(\pm,\pm)$\\
$\frak{g}_{2}(\mathbb{C})$ & $\{1,2\}$ &$(1,2)$ &&$(\pm,\pm),(i,i)$\\
\hline
\end{tabular}
\end{table}

\begin{table}[H]
\caption{\bf The case $|\Xi|=2$, more component of $\kappa_H$}
\label{tab4}
\begin{tabular}{|c|c|c|c|c|c|}
\hline
$\fg$ & $\Xi$ & $\kappa_H$ &$\gamma$&$J$\\
\hline
$\frak{sl}(3,\{\mathbb{R,C}\})$ & $\{1,2\}$ & $(1,2)$,$(2,1)$&&$(-,-)$\\
$\frak{sl}(3,\mathbb{C})$ & $\{1,2\}$ & $(1',2')$,$(1,2')$,$(2,2')$&&$(+,-)$\\
$\frak{sl}(4,\mathbb{R})$ & $\{1,2\}$& $(2,3)$,$(1,2)$,$(2,1)$  &&$(+,-)$\\
$\frak{sl}(4,\mathbb{C})$ & $\{1,3\}$& $(1,2)$,$(1,1')$,$(3,1')$&&$(-,+)$\\
$\frak{sl}(4,\mathbb{C})$ & $\{1,3\}$& $(3,2)$, $(1,3')$, $(3,3')$&&$(+,-)$\\
$\frak{sl}(n+1,\mathbb{R}), n>2$& $\{1,2\}$& $(1,2)$,$(2,1)$&&$(+,-)$\\
$\frak{sl}(n+1,\mathbb{C}), n>3$& $\{1,2\}$&$(1,2)$, $(2',1')$, $(1,1')$&&$(-,+)$\\
$\frak{sl}(n+1,\mathbb{C}), n>3$& $\{1,3\}$& $(1,2)$,$(1,1')$&&$(-,+)$\\
$\frak{sl}(n+1,\mathbb{C}), n>3$& $\{1,n\}$ & $(1,2)$, $(1,1') $,$(n,1') $ &&$(-,+)$\\
$\frak{sl}(n+1,\mathbb{R}), n>2$& $\{2,n\}$ & $(2,n)$,$(2,1)$&&$(-,+)$\\
$\frak{sl}(n+1,\mathbb{C}), n>3$& $\{2,n\}$& $(2,n)$,$(2,1)$,$(n,2')$&&$(-,+)$\\
\hline
$\frak{sl}(n+1,\mathbb{C})$& $\{1,p\}$&  $(1,p)$,$(1,p')$&&$(+,-)$\\
$n>3$&$p>2$&&&\\
$\frak{sl}(4,\mathbb{C})$ & $\{1,2\}$& $(2,3)$,$(1,2)$,$(2,1)$, &&$(+,-)$\\
&&$(1,2')$,$(1',2')$&&\\
$\frak{sl}(n+1,\mathbb{C})$& $\{1,2\}$ & $(1,2)$, $(2,1)$, &&$(+,-)$\\
$n>3$&&$(1,2')$, $(1',2')$&&\\
$\frak{sl}(n+1,\mathbb{C})$& $\{1,p\}$& $(1,2)$,$(1,1')$&$\alpha_p$&$(-,+)$\\
$n>3$&$p>3$&& $n-1>p$&\\
$\frak{sl}(n+1,\mathbb{C})$& $\{p,p+1\}$&  $ (p+1,p)$,$ ((p+1)',p')$&$\alpha_p$&$(-,+)$\\
$n>3$&$p>2$&&&\\
\hline
$\frak{so}(p,n-p)$ & $\{1,2\}$ & $(1,2)$,$(2,1)$&& $(-,+)$\\
$n>4,p>1$&&&&\\
\hline
$\frak{so}(n,\mathbb{C})$, $n>4$ & $\{1,2\}$ & $(1,2)$,$(2,1)$&& $(-,+)$\\
\hline
$\frak{sp}(4,\mathbb{C})$ & $\{1,2\}$ & $(1,2)$,$(2,1)$,&&$(-,+)$,
\\
&&$(1',2')$,$(1,2')$&&$(+,-)$\\
$\frak{sp}(2n,\mathbb{C})$ & $\{1,n\}$& $(1,2)$,$(1,1')$&$\alpha_n$&$(-,+)$\\
$n>2$&&&$ n>3$&\\
\hline
$\frak{sp}(6,\mathbb{C})$ & $\{2,3\}$&  $(2,3)$,$(2',3')$&&$(-,+)$\\
$\frak{sp}(2n,\mathbb{C}), n>2$ & $\{1,n\}$&   $(1,n)$,$(1,n')$&&$(+,-)$\\
\hline
\end{tabular}
\end{table}

\begin{table}[H]
\caption{\bf The case $|\Xi|=3,4$, one component of $\kappa_H$}
\label{tab5}
\begin{tabular}{|c|c|c|c|c|}
\hline
$\fg$& $\Xi$&$\kappa_H$& $\gamma$&$J$ \\
\hline
$\frak{sl}(n+1,\{\mathbb{R, C}\})$& $\{1,2,p\}$&  $(1,2)$&  $\alpha_p$&$(-,-,+)$,\\
& & &$n>p> 3$&$(+,-,+),(-,+,+)$\\
$\frak{sl}(n+1,\mathbb{R})$& $\{1,2,p\}$& $(2,1)$& $\alpha_1,\alpha_p$& $(-,-,-)$,\\
&&&$n>p> 3$&$(-,+,-)$, $(+,-,+)$\\
$\frak{sl}(n+1,\mathbb{C})$& $\{1,2,p\}$& $(2,1)$& $\alpha_1,\alpha_p$& $(-,-,-),(-,+,-)$, \\
&&&$n>p> 3$&$(+,-,+),(-i,i,i)$\\
$\frak{sl}(n+1,\{\mathbb{R, C}\})$& $\{1,p,n\}$& $(1,n)$& $\alpha_p$& $(-,-,+)$,\\
&&&$p>2$&$(+,-,-),(-,+,-)$\\
$\frak{su}(2,2)$& $\{1,2,3\}$& $(2,1)$& &$(-,-,-)$,\\
&&&&$(-,+,-),(+,-,+)$\\
$\frak{su}(n,n)$& $\{1,n,$&$(1,$&$\alpha_n$& $(-,+,-)$\\
&$2n-1\}$&$2n-1)$&$n>2$&\\
\hline
$\frak{su}(2,2)$& $\{1,2,3\}$& $(1,2)$&& $(+,-,+)$\\
\hline
$\frak{so}(2n,\mathbb{C}),$& $\{1,2,n\}$& $(1,2)$& $\alpha_2,\alpha_n$& $(-,-,+)$,\\
$\frak{so}(n,n)$, $n>3$&&&$n>4$&$(+,-,+),(-,+,+)$\\
$\frak{so}(3,5)$&$\{2,3,4\}$&$(4,2)$&&$(-,+,+)$\\
$\frak{so}(7,\mathbb{C})$,& $\{1,2,3\}$& $(3,2)$&& $(-,-,-)$,\\
$\frak{so}(3,4)$&&&&$(-,+,-),(+,-,+)$\\
\hline
$\frak{sp}(2n,\{\mathbb{R, C}\})$& $\{1,2,p\}$& $(2,1)$& $\alpha_1,\alpha_p$& $(+,-,-)$,\\
$n>3$&$p<n$&&&$(+,+,-),(+,-,+)$\\
$\frak{sp}(2n,\{\mathbb{R, C}\})$& $\{1,2,n\}$& $(2,1)$& $\alpha_1,\alpha_n$& $(-,-,-)$,\\
&&&$n>3$&$(-,+,-),(+,-,+)$\\
\hline
$\frak{sl}(n+1,\{\mathbb{R, C}\})$& $\{1,2,p,q\}$& $(2,1)$& $\alpha_1,\alpha_p,\alpha_q$& $(-,\pm,-,+)$,\\
&&&$n>4,p>3,$&$(+,\pm,-,-)$, \\
&&&$p\neq n, q\neq n$&$(-,\pm,+,-)$\\
\hline
$\frak{su}(n,n+1)$& $\{1,2,$& $(2,1)$&& $(-,+,+,-)$\\
&$2n-1,2n\}$&&&\\
\hline
\end{tabular}
\end{table}

\begin{table}[H]
\caption{\bf The case $|\Xi|=3,4$, more component of $\kappa_H$}
\label{tab6}
\begin{tabular}{|c|c|c|c|c|c|}
\hline
$\fg$ & $\Xi$ & $\kappa_H$ &$\gamma$&$J$\\
\hline
$\frak{sl}(4,\{\mathbb{R,C}\})$ & $\{1,2,3\}$&$(2,1)$,$(2,3)$,&& $(-,+,-)$\\
&&$(1,3)$&&\\
$\frak{sl}(4,\{\mathbb{R,C}\})$ & $\{1,2,3\}$&$(2,1)$,$(2,3)$,&& $(+,-,+)$\\
&&$(1,2)$,$(3,2)$&&\\
$\frak{sl}(n+1,\{\mathbb{R,C}\})$& $\{1,2,p\}$&$(1,2)$,$(2,1)$& $\alpha_p$&$(+,-,+)$\\
&&&$n>p>3$&\\
\hline
$\frak{sl}(4,\{\mathbb{R,C}\})$ & $\{1,2,3\}$&$(2,1)$,$(2,3)$&& $(-,-,-)$\\
$\frak{su}(2,2)$ & $\{1,2,3\}$ &$(1,3)$,$(2,1)$&&$(-,+,-)$\\
$\frak{su}(2,2)$ &  $\{1,2,3\}$&$(1,2)$,$(2,1)$&&$(+,-,+)$\\
$\frak{sl}(n+1,\{\mathbb{R,C}\})$& $\{1,2,n\}$&$(1,2)$,$(1,n)$&&$(-,-,+)$\\
$\frak{sl}(n+1,\mathbb{R})$& $\{1,2,n\}$&$(1,n)$,$(2,1)$&&$(-,+,-)$\\
\hline
$\frak{so}(4,4)$,$\frak{so}(8,\mathbb{C})$& $\{1,2,4\}$&$(1,2)$,$(4,2)$&$\alpha_2$&$(+,-,+)$\\
\hline
$\frak{sl}(n+1,\{\mathbb{R,C}\})$& $\{1,2,n-1,n\}$&$(2,1)$,$(n-1,n)$& &$(-,-,-,+)$,\\
&&&&$(+,-,-,-)$,\\
&&&&$(-,+,+,-)$\\
\hline
\end{tabular}
\end{table}


\newpage

\begin{thebibliography}{99}

\bibitem{parabook} \v Cap A., and J. Slov\'ak, 
``Parabolic Geometries I:
Background and General Theory'', Amer. Math. Soc., 2009.

\bibitem{BGG}\v Cap A., and J. Slov\' ak, V. Sou\v cek, \emph{Bernstein-Gelfand-Gelfand Sequences}, Ann. of Math. {\bf 154}, (2001), 97-113.

\bibitem{C1}\v Cap A., and K. Melnick, \emph{Essential Killing fields of parabolic geometries}, Indiana Univ. Math. J.  {\bf 62}, (2013), 1917-1953.

\bibitem{G2} Gregorovi\v c J., \emph{General construction of symmetric parabolic geometries}, Differential Geometry and its Applications {\bf 30}, (2012), 450-476 

\bibitem{G4} Gregorovi\v c J., ``Geometric structures invariant to symmetries'', Masaryk University, 2012.

\bibitem{G3} Gregorovi\v c J., \emph{Classification of invariant AHS--structures on semisimple locally symmetric spaces}, Central European journal of mathematics {\bf 11}, (2013), 2062-2075.

\bibitem{M} Hammerl M., \emph{Homogeneous Cartan geometries}, Archivum Mathematicum (Brno) {\bf 43}, (2007), 431-442.

\bibitem{Kn96}  Knapp A.W., ``Lie Groups Beyond an Introduction'', Birkh\"auser, 1996.
 
\bibitem{KMS} Kol\' a\v r I., and P.W. Michor, J. Slov\' ak, ``Natural operations in differential geometry'', Springer-Verlag, Berlin-Heidelberg-New~York, 1993.

\bibitem{KT} Kruglikov B., and D. The, \emph{The gap phenomena in parabolic geometries},  arXiv:1303.1307v4

\bibitem{L3} Loos O.,\emph{Spiegelungsraume und homogene symmetrische Raume}, Math. Z. {\bf 99}, (1967), 141-170. 

\bibitem{S} Sharpe R. W., ``Differential geometry: Cartan's generalization of
Klein's Erlangen program'', Springer--Verlag, 1997.

\bibitem{Si} \v Silhan J., \emph{A real analog of Kostant's version of the Bott-Borel-Weil theorem}, Journal of Lie Theory {\bf 14}, (2004), 481-499 \\Algorithm available on: \url{http://bart.math.muni.cz/~silhan/lie/}

\bibitem{Y} Yamaguchi K., \emph{Differential systems associated with simple graded Lie algebras}, Advanced Studies in Pure Mathematics {\bf 22}, Progress in Differential Geometry, (1993), 413–494.

\bibitem{Z1} Zalabov\' a L., \emph{Symmetries of Parabolic Geometries}, Differential Geometry and its Applications {\bf 27}, (2009), 605-622. 


\bibitem{Z2} Zalabov\' a L., \emph{Symmetries of Parabolic Contact Structures}, Journal of Geometry and Physics {\bf 60}, (2010), 1698-1709. 
\end{thebibliography}
\end{document}